\documentclass[12pt]{amsart}

[1994/12/01]

\usepackage{amsfonts}
\usepackage{amsmath}
\usepackage{geometry}
\usepackage{amssymb}
\usepackage{amsthm}
\usepackage{enumitem}
\usepackage{latexsym}
\usepackage{color}
\usepackage[english]{babel}
\usepackage{indentfirst}
\usepackage{fullpage}
\usepackage{enumitem}

\newcommand{\N}{\mathbb N}
\newcommand{\R}{\mathbb R}

\theoremstyle{definition}
\newtheorem{defin}{Definition}[section]

\newtheorem{lem}[defin]{Lemma}
\newtheorem{exam}[defin]{Example}
\newtheorem{teo}[defin]{Theorem}
\newtheorem{cor}[defin]{Corollary}
\newtheorem{obs}[defin]{Observation}

\newtheorem{question}[defin]{Question}

\newcommand{\eval}[2][\right]{\relax
	\ifx#1\right\relax \left.\fi#2#1\rvert}

\begin{document}

	\title[Oscillations]{Ramsey Property and Block Oscillation Stability on Normalized Sequences in Banach Spaces}	
	
	\author[S. Garcia-Ferreira]{S. Garcia-Ferreira}	
	\address{Centro de Ciencias Matem\'aticas, Universidad Nacional
Aut\'onoma de M\'exico, Apartado Postal 61-3, Santa Mar\'{\i}a,
58089, Morelia, Michoac\'an, M\'exico}
	\email{sgarcia@matmor.unam.mx}
	
	\author[A. C. Hernandez-Soto]{A. C. Hernandez-Soto}
	\address{Centro de Ciencias Matem\'aticas, Universidad Nacional
Aut\'onoma de M\'exico, Apartado Postal 61-3, Santa Mar\'{\i}a,
58089, Morelia, Michoac\'an, M\'exico}
	\email{caren.hdez@gmail.com}
	
	\thanks{ Research of the first-named author was supported by the PAPIIT grant no. IN-105318. }
	
	
	\date{}
	
	\commby{}
	
	
	\begin{abstract} 
	A well-known application of the Ramsey Theorem in the Banach Space Theory is the proof of the fact that every normalized basic sequence has a subsequence which generates a spreading model (the Brunel-Sucheston Theorem). Based on this application, as an intermediate step, we can talk about the notion of $(k,\varepsilon)-$oscillation stable sequence, which will be described and analyzed more generally in this article. Indeed, we introduce the notion $((\mathcal{B}_i)_{i=1}^k,\varepsilon)-$block oscillation stable sequence where $(\mathcal{B}_i)_{i=1}^k$ is a finite sequence of barriers and using what we will call blocks of barriers. In particular, we prove that the Ramsey Theorem is equivalent to the statement ``for every finite sequence $(\mathcal{B}_i)_{i=1}^k$ of barriers, every $\varepsilon>0$ and every normalized sequence $(x_i)_{i\in\N}$ there is a subsequence $(x_i)_{i\in M}$ that is  $((\mathcal{B}_i\cap\mathcal{P}(M))_{i=1}^k,\varepsilon)-$block oscillation stable'', where $\mathcal{P}(M)$ is the power set of the infinite set M.
	Besides, we introduce the $(\mathcal{B}_i)_{i\in\N}-$block asymptotic model of a normalized basic sequence where $(\mathcal{B}_i)_{i\in\N}$ is a sequence of barriers. These models are a generalization of the spreading models and are related to the $((\mathcal{B}_i)_{i=1}^k,\varepsilon)-$block oscillation stable sequences. We show that the Brunel-Sucheston is satisfied for the $(\mathcal{B}_i)_{i\in\N}-$block asymptotic models, and we also prove that this result is equivalent to the Ramsey Theorem. The difference between our theorem and the Brunel-Sucheston Theorem is based on the number of different models that are obtained from the same normalized basic sequence through them. This and other observations about $(\mathcal{B}_i)_{i\in\N}-$block asymptotic models are noted in an example at the end of the article.
	\end{abstract}
	
	\maketitle

\section{Introduction and Preliminaries}

The Ramsey Theorem \cite{ra} is a powerful tool on many fields of Mathematics. In particular, this combinatorial principle has had a strong impact on the study of Banach spaces in the last few decades (c.f. e.g., \cite{asym}, \cite{k}, \cite{Od}, \cite{Od2}). For example, J. Farahat uses a strong version of Ramsey's Theorem \cite{Fa} to prove the famous Rosenthal Theorem on $\ell_1$ \cite{Rose1}. Another remarkable application is due to A. Brunel and L. Sucheston, who showed that every normalized basic sequence in a Banach space has a subsequence that generates a spreading model \cite{bs}. Th. Schlumprecht presented a beautiful proof of this fact using the theorem known as Ramsey's Theorem for Analysts in \cite{Schlu}. This proof motivated the introduction of the notion of $(k,\varepsilon)-$oscillation stable sequence in the paper \cite{cg}. In this last unpublished article, it is observed that the Ramsey Theorem is equivalent to a statement which involves the property of $(k,\varepsilon)-$oscillation stability. Moreover, this new notion allows to prove that the Brunel-Sucheston Theorem implies the Ramsey Theorem.

\medskip
In this work, the Banach spaces that we will consider are those that have a Schauder basis, which we will also refer to as a normalized basic sequence.
To put in context the basic tools that we will use, we would like to mention that there exist families of finite subsets of $\N$ with suitable characteristics known as barriers, which had been useful in the construction of examples and counterexamples, and study of Banach spaces (see, for instance, \cite[Part B]{ss} and \cite{LT}). An example within this study is the use of barriers in the generation of certain block sequences of a normalized basic sequence, this application is based on the importance of block subsequence of a basic sequence in Banach Space Theory (see the Krivine Theorem and the Zipping Theorem). In this way, we will introduce the notion of
$((\mathcal{B})_{i=1}^k,\varepsilon)$-block oscillation stable sequence which uses blocks of a finite sequence of barriers
which is based on the idea of oscillation stability of the article \cite{cg}, and we will show that the Ramsey Theorem is equivalent to (among other conditions) the statement that assures that for every finite sequence $(\mathcal{B}_i)_{i=1}^k$ of barriers, every $\varepsilon>0$ and every normalized sequence $(x_i)_{i\in\N}$ there is a subsequence $(x_i)_{i\in M}$ that is  $((\mathcal{B}_i\cap\mathcal{P}(M))_{i=1}^k,\varepsilon)-$block oscillation stable. This new notion of oscillation stability motivated us to introduce the $(\mathcal{B}_i)_{i\in\N}-$block asymptotic models of a normalized basic sequence where $(\mathcal{B}_i)_{i\in\N}$ is a sequence of barriers, which  generalizes the Brunel-Sucheston spreading models. This allows us to establish an equivalence of the Ramsey Theorem within the context of the existence of block asymptotic models. Although the block asymptotic models are a generalization of the spreading models, there are differences between them, through an example we will see that not all block asymptotic models are spreading. 

\medskip

In the sequel, we introduce the notation and terminology that are needed.

Given an infinite set $X$, for each $\kappa\leq\omega$, we set $[X]^{\kappa} = \{ s \subseteq X : |s| = \kappa \}$, $[X]^{\leq \kappa} = \{ s \subseteq X : |s| \leq \kappa \}$ and $[X]^{<\kappa} = \{ s \subseteq X : |s| < \kappa \}$. For convenience it will be useful not to consider zero as an element of natural numbers. The increasing enumeration of an infinite set $M\in[\N]^{\omega}$ will be $\{m_i:i\in\N\}$. Also, if $M\in[\N]^{\omega}$ and $n\in\N$, then $M/n:=\{m\in M: m>n\}$. 

The symbol $FIN^*$ will denote the family of all finite nonempty subsets of $\N$. To specify the elements of $s\in FIN^*$ we will write $s=\{s(1),\ldots,s(|s|)\}$ and in this notation we will always assume that $s(1)<s(2)<\cdots<s(|s|)$. If $s, t \in FIN^*$, we say that ``$s$ is {\it less than} $t$'' and write $s<t$ if $\max (s) < \min (t)$, i.e., $s(|s|)<t(1)$. In the case when $s=\{n\}$ for some $n\in\N$, we simply write $n<t$. If $s,t\in FIN^*$, then $s\sqsubseteq t$ means that $s$ is a {\it initial segment} of $t$ and when $s<t$ we denote $s^{\frown}t=s\cup t$. If $s\in FIN^*$ and $M\in[\N]^{\omega}$, then  $s\sqsubseteq M$ also means that $s$ is a {\it initial segment} of $M$. 

\medskip

For each $i \in \N$, the symbol $e_i$ will denote the vector in $\R^{\N}$ defined by $e_i(j)= \delta_{i,j}$ (the Kronecker delta) for all $j\in\N$. The sequence $(e_i)_{i \in \N}$ will be called the {\it canonical basis} of $\R^{\N}$. For the finite case, the set $\{e_1,\cdots, e_k\}$ will denote the canonical basis of $\mathbb{R}^k$. For every $a = (a_i)_{i \in \N} \in \mathbb{R}^\N$ we define its support as $supp(a) = \{ i \in \N : a_i \neq 0 \}$. Recall that $c_{00}:=\{a=(a_i)_{i\in\N}\in\R^{\N}:supp(a)<\omega\}$ is a normed space under the supremum norm:
\begin{align*}
\|a\|_{\infty}=\sup\{|a_i|:i\in\N\}
\end{align*}
for each $a=(a_i)_{i\in\N}=\sum_{i=1}^na_ie_i\in c_{00}$.

\medskip

The required notation concerning Banach spaces is standard and can be found in books like \cite{ca}, \cite{di} and \cite{G-D}. However, for the sake of the non-expert, we recall some notation and definitions.

Our Banach spaces will be separable, infinite dimensional and real. In case that several Banach spaces are involved and we want to specify the norm of a Banach space $X$ we will write $\| \cdot \|_X$. For convenience, in some cases, a sub-sequence $(x_{m_i})_{i\in\N}$ of a sequence $(x_i)_{i\in\N}$ will be written as $(x_i)_{i\in M}$ where $M\in[\N]^{\omega}$. For every subset $Y$ of a Banach space $(X,\|\cdot\|)$, the symbol $\langle Y \rangle$ will denoted the span of $Y$. 
 
Let $(X,\|\cdot\|)$ be a Banach space and let $(x_i)_{i \in \mathbb{N}}$ be a sequence in $X$. The sequence $(x_i)_{i\in \mathbb{N}}$ is called {\it normalized} if $\|x_i\| = 1$ for all $i\in \mathbb{N}$. For each $s\in FIN^*$ we define the block vector $\mathcal{X}(s)=\frac{\sum_{i\in s}x_i}{\|\sum_{i\in s}x_i\|}$. 
We say that $(x_i)_{i \in \mathbb{N}}$ is a {\it Schauder Basis} for $X$ iff there is a unique sequence of real numbers $(a_i)_{i \in \mathbb{N}}$ such that $x = \sum_{i =1}^\infty a_ix_i$ for all $x\in X$. In general, we say that $(x_i)_{i \in \mathbb{N}}$ in $X$ is a basic sequence iff is a Schauder basis for the closure of its span. It is very convenient for us to use the following equivalence of Schauder basis condition \cite[Th. 3.2]{ca}: there is $C>0$ so that for all $m<n$ and $(a_i)_{i =1}^m \in[-1,1]^m$ we have $\|\sum_{i = 1}^m a_ix_i\| \leq C\|\sum_{i = 1}^n a_ix_i\|$, and the topological closure of its span coincides with the whole $X$. The smallest such $C$ is called the {\it basis constant} of $(x_i)_{i\in\N}$. Two basic sequences $(x_i)_{i\in\N}$ and $(y_i)_{i\in\N}$ are {\it equivalent} if there exist $A,B>0$ so that for all $n\in\N$ and scalars $(a_i)_{i=1}^n\in[-1,1]^n$ one has that
\begin{equation*}
A^{-1}\|\sum_{i=1}^ka_ix_i\|_X\leq\|\sum_{i=1}^ka_iy_i\|_Y\leq B\|\sum_{i=1}^ka_ix_i\|_X.
\end{equation*} 
For a basic sequence $(x_i)_{i\in\N}$ and a real sequence $(a_i)_{i\in\N}$, a sequence of non-zero vectors $(y_j)_{j\in\N}$ of the form
\begin{align*}
y_k=\sum_{\ell=p_k}^{p_{k+1}-1}a_{\ell}x_{\ell}
\end{align*}
where $p_1<p_2<\cdots<p_k<\cdots$ is an increasing sequence of natural numbers, is called a {\it block basis sequence}.

\medskip 

For the convenience of the reader, we believe it is convenient to keep in mind the Ramsey Theorem \footnote{A proof of this theorem lies in \cite{To} or  \cite{Od2}.} and the Brunel-Sucheston Theorem \footnote{For more properties and a proof of this theorem, consult \cite{ss} or \cite{Od2}.} that we enunciate below.

\begin{teo}\label{Teo:Ramsey}{\bf [Ramsey].} Let $N\in[\N]^{\omega}$ and $k,q\in\N$. For every function (finite coloring) $\varphi:[N]^{k} \longrightarrow\{1,\ldots,q\}$ there are $i\leq q$ and (a monochromatic set) $M \in [N]^{\omega}$ such that $[M]^k\subseteq \varphi^{-1}(i)$.
\end{teo}

As mentioned above, one of our purposes is to generalize the following notion through the use of barrier blocks, in addition, it is the main concept of the Brunel-Sucheston Theorem.

\begin{defin}\label{Def:ModeloDisperso}
	Let $(x_i)_{i \in \N}$ be a normalized basic sequence in a Banach Space $(X,\|\cdot\|_X)$. A Schauder basis $(y_i)_{i\in\N}$ for a Banach space $(Y,\|\cdot\|_{Y})$ is a {\it spreading model}\footnote{For a deeper presentation of the Theory of the spreading models, see \cite{Beau}} of $(x_i)_{i \in \N}$ if there is a sequence $\varepsilon_i\searrow 0$
	such that for every $s = \{s(1),\ldots, s(k)\} \in FIN^*$ with $s(1) \geq |s| = k$ we have that
	$$
	\big|\|\sum_{j = 1}^k a_jx_{{s(j)}} \|_X -\|\sum_{j = 1}^k a_jy_j \|_Y \big| < \varepsilon_{s(1)},
	$$
	for every $(a_i)_{i =1}^ k \in [-1,1]^k$. We also say that $(x_i)_{i \in \mathbb{N}}$ {\it generates} $(y_i)_{i \in \mathbb{N}}$ (or $Y$) as spreading model.	
\end{defin}

\begin{teo}{\bf [Brunel-Sucheston].}
\label{Teo:BruSu} 
Every normalized basic sequence has a subsequence that generates a spreading model.
\end{teo}

The article is organized as follow: In the second section, we recall the definition of a barrier and state some basic properties of barriers, in particular, some characteristics that relate to the barriers and their lexicographical rank. The third section is devoted to introduce and study the family of blocks of a finite sequence of barriers. In this section, we prove that these families have the Ramsey-type property and present the version for analysts of this property. In the fourth section, we introduce the notion of $((\mathcal{B}_i)_{i=1}^k,\varepsilon)-$block oscillation stable sequence and show that the Ramsey Theorem is equivalent to several statements, some of them involve the block oscillation stability. The fifth section concerns to the generalization of the Brunel-Sucheston Theorem via the $((\mathcal{B})_{i\in\N},\varepsilon)$-block asymptotic models. Finally, in the last section, we build an example of a normalized basic sequence that generates two equivalent block asymptotic models, only one of them is spreading and they are associated with distinct sequences of barriers. In addition, base on this example, we formulate open problems concerning about the existence of two non-equivalent block asymptotic models from the same normalized basic sequence by using two distinct sequences of barriers.

	\section{Barriers}

We start this section with notation that we will us throughout this paper: If $s \in FIN^*$, $M\in[\N]^{\omega}$ and $\mathcal{B}\subseteq FIN^*$, then
\begin{align*}
\mathcal{B}\upharpoonright_{M}:=\{s\in\mathcal{B}:s\subset M\}\\
\mathcal{B}_{s}=\{t\in FIN^*:s^{\frown} t\in\mathcal{B}\}.
\end{align*}

Now, we present a basic notion of the Nash-William Theory.

\begin{defin}\label{Def:Barrera}
	Given a set $N\in[\N]^{\omega}$, a nonempty family $\mathcal{B}$ of finite subsets of $\N$ is called a {\it barrier} on $N$ provided that:
	\begin{enumerate}[label={(B\arabic*)}]
		\item $\mathcal{B} \subseteq [N]^{< \omega}\setminus\{\emptyset\}$.\label{B1}
		\item For every $s,t\in \mathcal{B}$ if $s\neq t$ then $s\nsubseteq t$ and $t\nsubseteq s$.\label{Bcontencion}
		\item For every $M\in[N]^{\omega}$ there is $s\in\mathcal{B}$ such that $s\sqsubseteq M$.\label{Bsegmento}
	\end{enumerate}
\end{defin}

The simplest examples of barriers are:
\begin{enumerate}[label={\arabic*.}]
	\item For each $k\in \N$ and $N\in[\N]^{\omega}$, $[N]^k=\{s\subseteq N:|s|=k\}~\text{is a barrier on}~N.$
	\item The Schreier barrier $\mathcal{S}=\{s\in FIN^*: |s|=\min (s)\}$.
\end{enumerate}

We list in the first lemma of this section the basic properties of barriers that we will use.

\begin{lem}\label{Lem:PropiedadesBarrera}
\begin{enumerate}[label={\arabic*.}]
\item  Let $\mathcal{B}\subseteq FIN^{*}$ be a barrier on $N\in[\N]^{\omega}$.
	\begin{enumerate}
		\item For every $M\in[N]^{\omega}$, the restriction $\mathcal{B}\upharpoonright_{M}$ is a barrier on $M$.
		\item For each $s\in FIN^{*}\setminus\mathcal{B}$, the set $\mathcal{B}_s$ is a barrier on $N/\max(s)$.
		\item Every barrier $\mathcal{B}$ is associated to a barrier on $\N$: If $\{n_i:i\in\N\}$ are in ascending order, then $\mathcal{B}_{\N} :=\{\{i\in\N:n_i\in s\}:s\in\mathcal{B}\}$ is a barrier on $\N$.
	\end{enumerate}
\item If $\mathcal{B}_1$ and $\mathcal{B}_2$ are barriers on some $N\in[\N]^{\omega}$ so that $\mathcal{B}_1\subseteq\mathcal{B}_2$, then $\mathcal{B}_1=\mathcal{B}_2$.

\item Let $k\in\N$ and let $\mathcal{B}_1,\ldots,\mathcal{B}_k\subseteq FIN^{*}$ be barriers on $N\in[\N]^{\omega}$. Then the set $$\bigoplus^k_{i=1}\mathcal{B}_i:=\{s_1^{\frown}\cdots^{\frown}s_k:s_i\in\mathcal{B}_i~ \forall i\leq k ~\text{and}~ s_1<\cdots<s_k\}$$ is a barrier on $N$.
\end{enumerate}
\end{lem}

Association 1.(c) allows us to assume that all our barriers could be taken on $\N$ as the case suits us, although in some situations we will not consider it.

\medskip

The use of notion of ``$\alpha-$uniform barrier'' in inductive arguments to establish barrier properties is common. To avoid the use of this concept we will obtain some characteristics of barriers by employing a more basic notion, known as ``the lexicographic order'' of a barrier. For our purposes, we basically need Corollary \ref{Cor:rankBar} which will prove without the use of 
uniform barriers.
\medskip

{\it The lexicographic order} on $FIN^*$ is defined by $s<_{lex}t$~ whenever $\min(s\bigtriangleup t)\in s$, where $s\bigtriangleup t$ is the symmetric difference $s\backslash t\cup t\backslash s$, for every $s,t\in FIN^*$.

\medskip

Next we state the must outstanding properties of the lexicographic order of a barrier (for a proof see \cite[L. II.2.15]{ss}).

\begin{lem}\label{Lem:OrdenLexicograficoBarrera} 
Every barrier is a well-ordered set with respect to the lexicographic order. A based on this, the order type of lexicographic order of a barrier $\mathcal{B}$ is denoted by $rank(\mathcal{B})$ and is named {\it the lexicographical rank of $\mathcal{B}$}.
\end{lem}

Given a barrier $\mathcal{B}$ on $N\in[\N]^{\omega}$, we remark that the lexicographical rank of $\mathcal{B}$ coincides with the lexicographical rank of its associated barrier on $\N$ (clause 1.(c) of Lemma \ref{Lem:PropiedadesBarrera}).

\medskip

Using  elementary properties of well-ordered sets we get the following result.

\begin{lem}\label{Lem:rankBn}
Let $\mathcal{B}$ be a barrier on $\N$. Then $rank(\mathcal{B}_{\{n\}})<rank(\mathcal{B})$ for all $n\in\N$.
\end{lem}

To establish the important results of this section (Theorem \ref{Teo:BarOmk} and Corollary \ref{Cor:rankBar}) and that will be used in the following ones, we need to state and prove the following three preliminary lemmas.

\begin{lem}\label{Lem:s-rankB}
	Every barrier $\mathcal{B}$ satisfies that $rank(\mathcal{B})\geq\omega^{|s|}$ for all $s\in\mathcal{B}$.
\end{lem}

\begin{proof}
	Suppose without loss of generality that $\mathcal{B}$ is a barrier on $\N$ and let $s\in\mathcal{B}$. It is enough to prove that $\mathcal{B}$ has a subset $\mathcal{F}$ such that $s\leq_{lex}t$ for all $t\in \mathcal{F}$ and its lexicographical rank is $\omega^k$ where $k=|s|$. We proceed by induction on $k$:
	For the case $k=1$, put $s=\{s(1)\}\in\mathcal{B}$. By definition of barrier, for every $i\in\N$ we get $s_i\in\mathcal{B}$ such that $$s_i\sqsubseteq\{s(1)+i-1,s(1)+i,s(1)+i+1,s(1)+i+2,\ldots\}.$$ Notice that $s_i<_{lex}s_{i+1}$ for all $i\in\N$, by construction. Hence, there is an order isomorphism from $\mathcal{F}:=\{s_i:i\in\N\}$ to $\omega$, which is induced naturally by lexicographic order. This shows that $rank(\mathcal{F})=\omega$.
	For $k=2$, we set $s=\{s(1),s(2)\}\in\mathcal{B}$ and $s_1:=s$. By property \ref{Bsegmento}, it is possible to find $s_i\in\mathcal{B}$ such that $$s_i\sqsubseteq\{s(2)+i-2,s(2)+i-1,s(2)+i,s(2)+i+1,\ldots\}$$ for each $i\in\N/1$. The property \ref{Bcontencion} implies that $\{s(2)+i-2,s(2)+i-1\}\sqsubseteq s_i$ for every $i\in\N/1$. Thus $|s_i|\geq 2$ and $s_i<_{lex}s_{i+1}$ for all $i\in\N$. Similarly, for each $i,j\in\N$ there is an $s_{i,j}\in\mathcal{B}$ so that $s_{i,j}\sqsubseteq \{s_i(1),\ldots s_i(n_i-1),s_i(n_i)+j,s_i(n_i)+j+1,\ldots\}$ and $\{s_i(1),\ldots s_i(
	n_i-1),s_i(n_i)+j\}\sqsubseteq s_{ij}$, where $n_i:=|s_i|$. It follows from construction that $s_i<_{lex}s_{i,1}$ and $s_{i,j}<_{lex}s_{i,j+1}<_{lex}s_{i+1}$ for every $i,j\in\N$. This shows that there is an order isomorphism from $\mathcal{F}:=\{s_i:i\in\N\}\cup\{s_{i,j}:i,j\in\N\}$ to $\omega^2$ given by the lexicographic order.
	
	\medskip
	 
	Suppose now that if  $s\in\mathcal{B}$ and $|s|=k$, then there is a family $\mathcal{F}\subseteq\mathcal{B}$ with lexicographical rank equal to $\omega^k$ such that $s\leq_{lex}t$ for all $t\in \mathcal{F}$. To proceed to the inductive step $k+1$, we set $s=\{s(1),\ldots,s(k+1)\}\in\mathcal{B}$. Since $\mathcal{B}$ is a barrier on $\N$, for every $i\in\N$. there is an element $s_i\in\mathcal{B}$ such that
	\begin{align*}
		s_i\sqsubseteq\left\{
		\begin{array}{ll}
		\{s(i),s(i+1),\ldots s(k),s(k+1),s(k+1)+1,\ldots\}&\text{if}~i\leq k+1\\
		\{s(k+1)+(i-k)-1,s(k+1)+(i-k),&\\
		~~s(k+1)+(i-k)+1,s(k+1)+(i-k)+2,\ldots\}&\text{if}~ i>k+1
		\end{array}
		\right..
	\end{align*}
	Also, condition \ref{Bcontencion} of barriers implies that
	\begin{align*}
		&\{s(i),s(i+1),\ldots s(k),s(k+1),s(k+1)+1,\ldots,\\
		&\qquad\qquad s_1(k+1)+(i-1)\}\sqsubseteq s_i~\text{if}~i\leq k+1,\\
		\text{o}~&\left\{s(k+1)+(i-k)-1,s(k+1)+(i-k),s(k+1)+(i-k)+1,\ldots,\right.\\
		&\qquad\qquad \left. s(k+1)+(i-k)+(k-1)\right\}\sqsubseteq s_i~\text{if}~i> k+1
	\end{align*}
	for each $i\in \N$. Thus we have that $|s_i|\geq k+1$ and $s_i<_{lex}s_{i+1}$ for all $i\in\N$. After renaming, for every $i\in\N$, the elements of $s_i$ as $\{s_i(1),\ldots,s_i(n_i)\}$ where $|s_i|=n_i\geq k+1$, we will consider the barrier $\mathcal{B}_{\{s_i(1),\ldots,s_i({n}_i-k)\}}$, and take the element $$t_i:=\{s_i(n_i-k+1),s_{i}(n_i-k+2),\ldots,s_{i}(n_i)\}\in\mathcal{B}_{\{s_i(1),\ldots,s_i({n}_i-k)\}}.$$ Since $|t_i|=k$ for each $i\in\N$, the inductive hypothesis says that there is a set $\mathcal{G}_i\subseteq\mathcal{B}_{\{s_i(1),\ldots,s_i({n}_i-k)\}}$ with lexicographical rank equal to ${\omega}^k$ such that $t_i\leq_{lex}t$ for all $t\in \mathcal{G}_i$. Let
	\begin{equation*}
		\mathcal{F}_i=\big\{\{s_i(1),\ldots,s_i({n}_i-k)\}^{\frown}t:t\in \mathcal{G}_i\big\}\subseteq\mathcal{B}
	\end{equation*}
	for each $i\in\N$. For every $i\in\N$ the function $f_i:\mathcal{G}_i\rightarrow \mathcal{F}_i$ defined by $$f_i(t)=\{s_i(1),\ldots,s_i({n}_i-k)\}^{\frown}t$$ for each $t\in \mathcal{G}_i$ is an order isomorphism. Hence $rank(\mathcal{F}_i)=\omega^k$ for all $i\in\N$. Notice that  $s<_{lex}s_{i+1}$ for every $i\in\N$ and $s\in S_i$ (since $s(1)<s_{i+1}(1)$). This induces the existence of an order isomorphism from $\mathcal{F}=\cup_{i\in\N} \mathcal{F}_i$  to $\omega^{k+1}$, i.e., $rank(\mathcal{F})=\omega^{k+1}$.
\end{proof}

We know that $[\N]^1$ is the only $1-$uniform barrier on $\N$ (see \cite[Ex.II.3.5]{ss}) and in our context we establish this fact in another way:

\begin{cor}\label{Cor:rank=omega}
 The family $[\N]^1$ is the only barrier on $\N$ of lexicographical rank equal to $\omega$.
\end{cor}

The next lemma provides some properties about barriers that are different to $[\N]^1$.

\begin{lem}\label{Lem:rank>omega}
If $\mathcal{B}$ is a barrier on $\N$ and $\mathcal{B}\neq [\N]^{1}$, then $|\mathcal{B}\cap[\N]^{1}|<\omega$. As consequence, one has that $\{m\}\in\mathcal{B}$ for each $m\leq n$ and $\mathcal{B}\backslash[\N]^1$ is a barrier on $\N/n$ where $n:=\max \big(\bigcup(\mathcal{B}\cap[\N]^{1})\big)$.
\end{lem}

\begin{proof}
	Let $\mathcal{B}$ be a barrier on $\N$ such that $\mathcal{B}\neq[\N]^{1}$ and $\mathcal{B}\cap[\N]^{1}$ is an infinite set. Take $s=\{s(1),\ldots,s(k)\}\in\mathcal{B}\backslash[\N]^{1}$ with $k\geq 2$. By property \ref{Bsegmento} of barriers, for each $i\geq s(k)$ we get $t_i\in\mathcal{B}$ such that $t_i\sqsubseteq\{s(1),\ldots,s(k-1),i,i+1,\ldots\}$, and from property \ref{Bcontencion} it follows that $\{s(1),\ldots,s(k-1),i\}\sqsubseteq t_i$.
	Besides, since $\mathcal{B}\cap[\N]^{1}$ is infinite, there is a number $m\in\N$ so that $m>s(k)$ and $\{m\}\in\mathcal{B}$. Thus we have the contradiction $\{m\}\subseteq t_m$.
	
	Let $n:=\max \big(\bigcup (\mathcal{B}\cap[\N]^{1})\big)$ and suppose there is a number $m<n$ such that $\{m\}\notin\mathcal{B}$. Consider the infinite set $\{m,n,n+1,n+2,\ldots\}\subseteq\N$. By definition of barrier, there is an element $s\in\mathcal{B}$ such that $s\sqsubseteq\{m,n,n+1,n+2,\ldots\}$. We know that $s\neq\{m\}$, then $\{m,n\}\sqsubseteq s$ and so $\{n\}\subseteq s$. This contradiction completes the proof.
\end{proof}

Now we pay attention to the first main result.

\begin{teo}\label{Teo:BarOmk}
For each $k\in\N$ with $k\geq 2$ and barrier $\mathcal{B}$ on $\N$ we have the following assertions:
	\begin{enumerate}
		\item[${(1)}_k$]\label{BOKEst} The lexicographical rank of $\mathcal{B}$ is $\omega^k$ if only if $|s|\leq k$ for all $s\in\mathcal{B}$ and there exists an $n\in\N\cup\{0\}$ such that $[\N/n]^k\subseteq\mathcal{B}$.\\
		\item[${(2)}_k$]\label{BOKInt} If $rank(\mathcal{B})>\omega^k$, then the set
		\begin{equation*}
		\mathfrak{M}_k(\mathcal{B})=\big\{m\in\N:m=\min (s)~\text{for some}~s\in\mathcal{B}\cap[\N]^k\big\}~\text{is finite.}
		\end{equation*}
	\end{enumerate}
\end{teo}

\begin{proof}
	We proceed by induction on $k$. When $k=2$ we have that:
	\begin{itemize}
		\item[$(1)_2$] $(\Rightarrow)$ If there is an $s\in\mathcal{B}$ with $|s|>2$, then by Lemma \ref{Lem:s-rankB} we know that $rank(\mathcal{B})\geq\omega^3$. Therefore, $|s|\leq 2$ for all $s\in\mathcal{B}$. From this and Lemma \ref{Lem:rank>omega} it follows that $[\N/n]^2\subseteq\mathcal{B}$ where $n:=\max \big(\bigcup(\mathcal{B}\cap[\N]^{1})\big)$.\\
		$(\Leftarrow)$ Suppose that $|s|\leq 2$ for all $s\in\mathcal{B}$ and $[\N/n]^2\subseteq\mathcal{B}$ for some $n\in\N$. By hypothesis, the set $\mathcal{B}_{\{i\}}\subseteq[\N]^{<2}$ for each $i\leq n$, in other words, if $i\leq n$, then $\mathcal{B}_{\{i\}}$ is either a empty set or $\mathcal{B}_{\{i\}}=[\N/i]^1$. We assume, without loss of generality, that $\mathcal{B}_{\{i\}}\neq\emptyset$ for each $i\leq n$. This implies that the lexicographical rank of $\Gamma_{\{i\}}:=\{\{i\}^{\frown}s:s\in\mathcal{B}_{\{i\}}\}\subseteq\mathcal{B}$ is equal to  $rank(\mathcal{B}_{\{i\}})=\omega$ for every $i\leq n$. By properties of well-ordered sets and $\mathcal{B}=\{\Gamma_{\{i\}}\}_{i=1}^n\cup\{[\N/n]^2\}$, we get that
		\begin{equation*}
		rank(\mathcal{B})=\sum_{i=1}^n rank(\Gamma_{\{i\}})+rank([\N/n]^2)=\omega^2.
		\end{equation*}
		\item[$(2)_2$] Suppose that $rank(\mathcal{B})>\omega^2$ and $\mathfrak{M}_2(\mathcal{B})$ is an infinite set. First, we show that there is $\ell\in \N$ so that $\mathcal{B}_{\{\ell\}}$  has an element $s$ with $|s|\geq 2$. To prove this claim, we assume that $\mathcal{B}_{\{i\}}\subseteq[\N]^{<2}$ for each $i\in\N$. We know that $\mathcal{B}_{\{i\}}$ is the empty set or a barrier. By Lemma \ref{Lem:rank>omega}, without loss of generality we suppose that $\mathcal{B}_{\{i\}}\neq\emptyset$ for every $i\in\N$. Hence, $rank(\mathcal{B}_{\{i\}})=\omega$ for all $i\in\N$. From these equalities and some facts about well-ordered sets we obtain
		\begin{align*}
		rank(\mathcal{B})=\sum_{i=1}^{\infty}rank(\Gamma_{\{i\}})
		=\sum_{i=1}^{\infty}\omega=\omega^2,
		\end{align*}
		since, in the same way as the previous case, it is true that $rank(\Gamma_{\{i\}})=rank(\mathcal{B}_{\{i\}})$ for all $i\in\N $, but which contradicts the hypothesis. So we can choose $\ell\in\N$ such that there is $s\in\mathcal{B}_{\{\ell\}}$ with $|s|\geq 2$. By Lemma \ref{Lem:s-rankB}, we get that $rank(\mathcal{B}_{\ell})\geq\omega^2$ and from Lemma \ref{Lem:rank>omega} follows that the set $\mathfrak{M}_1(\mathcal{B}_{\{\ell\}})=\bigcup(\mathcal{B}_{\{\ell\}}\cap[\N]^{1})$ is finite. Therefore, if $n:=\max\mathfrak{M}_1(\mathcal{B}_{\{\ell\}})$, then $\mathcal{B}_{\{\ell\}}\upharpoonright_{\N/n}$ is a barrier on  $\N/n$ that has no elements with cardinality equal to $1$. On the other hand, since $\mathfrak{M}_2(\mathcal{B})$ is an infinite set, we find $m\in\mathfrak{M}_2(\mathcal{B})$ such that $m>n$. We take $i>m$ such that   $\{m,i\}\in\mathcal{B}\cap[\N]^{2}$ and consider the set $\{m,i,i+1,i+2,\ldots\}\subseteq\N/n$. By definition of barrier, we obtain $s\in\mathcal{B}_{\{\ell\}}$ such that $s\sqsubseteq \{m,i,i+1,i+2,\ldots\}$. Also, we know that $s\neq\{m\}$ because of the construction, so $\{m,i\}\sqsubseteq s$. Hence, ${\{\ell\}}^{\frown}s\in\mathcal{B}$ and $\{m,i\}\in\mathcal{B}$, which is a contradiction of property \ref{Bcontencion}.
	\end{itemize}

	Suppose that $(1)_j$ and $(2)_j$ hold for each $1<j\leq k$. Now, we will prove that the statements are valid for $k+1$.
	\begin{enumerate}
		\item[$(1)_{k+1}$] $(\Rightarrow)$ If there is $s\in\mathcal{B}$ with $|s|>k+1$, then by Lemma \ref{Lem:s-rankB} one gets that $rank(\mathcal{B})>\omega^{k+2}$. Therefore, $|s|\leq k+1$ for all $s\in\mathcal{B}$. Notice that $\mathfrak{M}_1(\mathcal{B}),\mathfrak{M}_2(\mathcal{B}),\ldots,\mathfrak{M}_k(\mathcal{B})$ are finite sets because of Lemma \ref{Lem:rank>omega} and $(2)_{j}$ for $2\leq j\leq k$. Hence $\mathfrak{M}_{k+1}(\mathcal{B})$ is a infinite set, which implies that $[\N/n]^{k+1}\subseteq\mathcal{B}$ where $n:=\max \big(\bigcup_{j\leq k}\mathfrak{M}_j(\mathcal{B})\big)$.\\
		$(\Leftarrow)$ Suppose that $|s|\leq k+1$ for all $s\in\mathcal{B}$ and $[\N/n]^{k+1}\subseteq\mathcal{B}$ for some $n\in\N$. By assumption, $\mathcal{B}_{\{i\}}\subseteq [\N]^{<k+1}$ for every $i\leq n$. Without loss of generality, we may assume that $\mathcal{B}_{\{i\}}\neq\emptyset$ for every $i\leq n$. From Lemma \ref{Lem:s-rankB} for each $i\leq n$ we have that $rank(\mathcal{B}_{\{i\}})\geq\omega^{d_i}$ where $d_i=\max\{|t|:t\in\mathcal{B}_{\{i\}}\}$. Notice that $d_i\leq k$ for all $i\leq n$, so $\omega^{d_i}\leq\omega^{k}$. Applying Lemma \ref{Lem:rank>omega} and $(2)_j$ for all $j<d_i$ we obtain that $\mathfrak{M}_1(\mathcal{B}_{\{i\}}),\ldots,\mathfrak{M}_{d_i-1}(\mathcal{B}_{\{i\}})$ are finite sets whenever $i\leq n.$ Therefore, for each $i\leq n$ the set $\mathfrak{M}_{d_i}(\mathcal{B}_{\{i\}})$ is infinite. For every $i\leq n$ we have that $rank(\mathcal{B}_{\{i\}})\leq\omega^{d_i}$ because of the contrapositive of $(2)_{d_i}$. Thus, we conclude that $rank(\mathcal{B}_{\{i\}})=\omega^{d_i}$ if $i\leq n$. Finally, by some properties of well-ordered sets and the equality $rank(\mathcal{B}_{\{i\}})=rank(\Gamma_{\{i\}})$ for each $i\leq n$, we get that
		\begin{align*}
			rank(\mathcal{B})&=\sum_{i\leq n}rank(\Gamma_{\{i\}})+ rank([\N/n]^{k+1})\\
			&=\sum_{i\leq n}\omega^{d_i}+\omega^{k+1}
			=\omega^{k+1}.
		\end{align*}
		\item[$(2)_{k+1}$] Suppose that $rank(\mathcal{B})>\omega^{k+1}$ and $\mathfrak{M}_{k+1}(\mathcal{B})$ is an infinite set.
		We first show the existence of $\ell\in\N$ such that there is at least one element $s$ in $\mathcal{B}_{\ell}$ with $|s|>k+1$. Assume that $\mathcal{B}_{\{i\}}\subseteq [\N]^{<k+1}$ for all $i\in\N$. From Lemma \ref{Lem:s-rankB} we get that $rank(\mathcal{B})\geq\omega^{d}$ where $d=\max\{|s|:s\in\mathcal{B}\}\leq k+1$.
		Notice that $|s|\leq d$ for all $s\in\mathcal{B}$ and the families $\mathfrak{M}_{1} (\mathcal{B}),\ldots,  \mathfrak{M}_{d-1}(\mathcal{B})$ have finite cardinality according to Lemma \ref{Lem:rank>omega} and clauses $(2)_2,\ldots(2)_{d-1}$. Hence  $\mathfrak{M}_{d}(\mathcal{B})$ is an infinite set. From this, we deduce that $[\N/n']^{d}\subseteq\mathcal{B}$ with $n':=\max\big(\cup_{i\leq d-1}\mathfrak{M}_i(\mathcal{B})\big)$.
		As we can suppose that $(1)_{k+1}$ holds, we conclude that $rank(\mathcal{B})=\omega^d\leq\omega^{k+1}$, which is a contradiction.
		So we find $\ell\in\N$ such that there exists $s\in\mathcal{B}$ with $|s|\geq k+1$. Then, from Lemma \ref{Lem:s-rankB} we get that $rank(\mathcal{B}_{\{\ell\}})\geq \omega^{k+1}$. The Lemma \ref{Lem:rank>omega} and the clauses $(2)_2,\ldots,(2)_k$ assure that $\mathfrak{M}_1(\mathcal{B}_{\{\ell\}}),\ldots,\mathfrak{M}_k(\mathcal{B}_{\{\ell\}})$ are finite sets. Hence $\mathcal{B}_{\{\ell\}}\upharpoonright_{\N/n}$ is a barrier on $\N/n$ such that $|s|\geq k+1$ for all $s\in\mathcal{B}_{\{\ell\}}\upharpoonright_{\N/n}$ where $n:=\max\big(\cup_{i\leq k}\mathfrak{M}_i(\mathcal{B}_{\{\ell\}})\big)>\ell$.
		Since $\mathfrak{M}_{k+1}(\mathcal{B})$ is an infinite set, there is $m\in\mathfrak{M}_{k+1}(\mathcal{B})$ such that $m>n$. We take $i_1,\ldots,i_k>m$ such that $i_1<\ldots<i_k$ and $\{m,i_1,i_2,\ldots,i_{k}\}\in\mathcal{B}\cap[\N]^{k+1}$, and consider the infinite set $\{m,i_1,i_2,\ldots,i_{k},i_{k}+1,i_{k}+2,\ldots \}\subseteq\N/n$. By definition of barrier, we find  $s\in\mathcal{B}_{\{\ell\}}\upharpoonright_{\N/n}$ such that $s\sqsubseteq \{m,i_1,i_2,\ldots,i_{k},i_{k}+1,i_{k}+2,\ldots\}$. But since $|s|\geq k+1$, it follows that $\{m,i_1,i_2,\ldots,i_{k}\}\sqsubseteq s$. Thus ${\{\ell\}}^{\frown}s$ and $\{m,i_1,i_2,\ldots,i_{k}\}$ belong to $\mathcal{B}$, this contradicts the property \ref{Bcontencion} of barriers.
	\end{enumerate}
\end{proof}

The following corollary is useful at starting inductive arguments.

\begin{cor}\label{Cor:rankBar}
	If $\mathcal{B}$ is a barrier on $\N$ with $rank(\mathcal{B})<\omega^{\omega}$, then $rank(\mathcal{B})=\omega^k$ for some $k\in\N$. 
\end{cor}

\begin{proof}
	Assume that $\omega^{k}\neq rank(\mathcal{B})<\omega^{\omega}$ for all $k\in\N$.
	Then there is $n\in\N$ such that $\omega^n<rank(\mathcal{B})<\omega^{n+1}$. From Lemma \ref{Lem:s-rankB} follows that $|s|\leq n$ for all $s\in\mathcal{B}$. This implies that $\mathfrak{M}_i(\mathcal{B})$ is a infinite set for some $i\leq n$. However, according to Lema \ref{Lem:rank>omega} and Theorem \ref{Teo:BarOmk}, the sets $\mathfrak{M}_1(\mathcal{B}),\ldots,\mathfrak{M}_n(\mathcal{B})$ are finite.
\end{proof}

\section{Ramsey Type Theorems}

We start this section with an important property of the Nash-Williams Theory that has a notable relationship with the Ramsey Theorem.

\begin{defin}{\bf [Nash-Williams].}
	A family $\mathcal{F}\subseteq FIN^*$ has {\it the Ramsey property} if for every coloring $\varphi:\mathcal{F}\longrightarrow \{1,\ldots,q\}$, where $q\in\N$, there is a set  $M\in[\N]^{\omega}$ such that at most one of the restrictions $$\varphi^{-1}(1)\upharpoonright_{M},\varphi^{-1}(2)\upharpoonright_{M},\ldots,\varphi^{-1}(q)\upharpoonright_{M}$$ is non-empty. In particular, $\mathcal{F}$ has {\it the non-empty Ramsey property} if there exists $M\in[\N]^{\omega}$ such that just one of the restrictions is non-empty.
\end{defin}

A proof of the well-known Ramsey property for Barriers lies in the book \cite[L. II.2.7]{ss}:

\begin{teo}{\bf [Ramsey Theorem on Barriers].} \label{Teo:BarrerasRamseyNoNula}
Every barrier has the non-empty Ramsey property.
\end{teo}

Our first objective is to generalize the previous theorem to blocks of barriers, which have been motivated by the relevance of the block sequences of a basic sequence in a Banach space. First of all, we introduce the notion of blocks of barriers.

\begin{defin}
Given $k\in\N$ and barriers $\mathcal{B}_1,\ldots,\mathcal{B}_k $ on $N\in[\N]^{\omega}$, we define the set
	\begin{equation*}
	{Bl}(\mathcal{B}_1,\ldots,\mathcal{B}_k):=\left\{\{s_1,\ldots,s_k\} : \forall i\leq k~(s_i\in\mathcal{B}_i)~\text{y}~ s_1<s_2<\cdots<s_k\right\},
	\end{equation*}
and we refer to it as a {\it family of blocks of barriers in $(\mathcal{B}_i)_{i=1}^k$}. If $\mathcal{B}_1=\mathcal{B}_2=\cdots=\mathcal{B}_k$, then we denote by ${Bl}^k(\mathcal{B}_1)$ the set ${Bl}(\mathcal{B}_1,\ldots,\mathcal{B}_k)$. Further, the elements of $Bl(\mathcal{B}_1,\ldots,\mathcal{B}_k)$ are called $k-$blocks. 
\end{defin}

The desire generalization is the following.

\begin{teo}{\bf [Ramsey Theorem on Blocks of Barriers].} \label{Teo:RamseyBBarreras}
Let $(\mathcal{B}_i)_{i=1}^k$ be a finite sequence of barriers on $N\in[\N]^{\omega}$ and let $q\in\N$. Then for every coloring
 $\varphi:{Bl}(\mathcal{B}_1,\ldots,\mathcal{B}_k)\rightarrow\{1,\ldots,q\}$, there are $i\leq q$ and $M\in[N]^{\omega}$ such that ${Bl}(\mathcal{B}_1\upharpoonright_{M},\ldots,\mathcal{B}_k\upharpoonright_{M})\subseteq\varphi^{-1}(i)$.
\end{teo}

\begin{proof} 
	We consider the function $f:{Bl}(\mathcal{B}_1,\ldots,\mathcal{B}_k)\rightarrow\bigoplus^k_{i=1}\mathcal{B}_i$ defined by $f(\{s_1,\ldots,s_k\})={s_1}^{\frown}{s_2}^{\frown}{\cdots}^{\frown}s_k$ for each $\{s_1,\ldots,s_k\}\in{Bl}(\mathcal{B}_1,\ldots,\mathcal{B}_k)$. It is easy to verify that $f$ is a bijection, which implies the existence of its inverse $f^{-1}$. Now, fix a finite coloring $\varphi:{Bl}(\mathcal{B}_1,\ldots,\mathcal{B}_k)\longrightarrow \{1,2,\ldots,q\}$ and consider the coloring $\varphi\circ f^{-1}:\bigoplus^k_{i=1}\mathcal{B}_i\rightarrow \{1,\ldots,q\}$. From Lemma \ref{Lem:PropiedadesBarrera}$(3)$ and Theorem \ref{Teo:BarrerasRamseyNoNula} we find $i\leq q$ and $M\in[N]^{\omega}$ so that
	\begin{equation*}
	 (\bigoplus^k_{i=1}\mathcal{B}_i)\upharpoonright_M\subseteq(\varphi\circ f^{-1})^{-1}(i).
	\end{equation*}
	Since $(\bigoplus^k_{i=1}\mathcal{B}_i)\upharpoonright_M=\bigoplus^k_{i=1}(\mathcal{B}_i\upharpoonright_M)$ and $f$ is a bijection, we get that ${Bl}(\mathcal{B}_1\upharpoonright_{M},\ldots,\mathcal{B}_k\upharpoonright_{M})\subseteq\varphi^{-1}(i)$.
\end{proof}

In what follows, we will study the version for Analysts  of Theorem \ref{Teo:RamseyBBarreras} as it was done in the paper \cite[Th. 2.2.1]{Schlu} for the Ramsey Theorem.

\begin{teo}{\bf [Ramsey Theorem on Blocks of Barriers for Analysts].}
\label{Teo:RamseyAnaBBarreras}
Let $(X,d)$ be a totally bounded metric space, let $k\in \N$ and let $(\mathcal{B}_i)_{i=1}^k$ be a sequence of barriers on $N\in[\N]^{\omega}$. For every function $Bl(\mathcal{B}_1,\ldots,\mathcal{B}_k) \longrightarrow X$ and $\varepsilon > 0$ there is $M \in  [N]^{\omega}$ such that
	\begin{equation*}
		d(\varphi(S),\varphi(T))<\varepsilon
	\end{equation*}
for each $S$, $T\in Bl(\mathcal{B}_1\upharpoonright_M,\ldots,\mathcal{B}_k\upharpoonright_M)$.
\end{teo}

\begin{proof}
	The basic idea of the proof is the same as the one presented in the article \cite{Schlu}, which consists of coloring by finite covers.
\end{proof}

It is worth mentioning that the particular case when 
the sequence consists of a single barrier that is $[\N]^k $ (or the sequence $([\N]^1,\ldots,[\N]^1)$ of length $k$) and the metric space is compact is equivalent to the Ramsey Theorem for Analysts. It is for this reason that we decided to name Theorem \ref{Teo:RamseyAnaBBarreras} as the Ramsey Theorem on Blocks of Barriers for Analysts.

\medskip

The version of asymptoticity that is implicit in the article \cite[Th. 2.2.1]{Schlu} is reformulated next. 
Given the importance of the inductive argument used in the demonstration of this reformulation, we include a complete proof.

\begin{cor}\label{Cor:RamseyAnaBBarrera}
Let $(X,d)$ be a totally bounded metric space, let $k\in \N$ and let $(\mathcal{B}_i)_{i=1}^k$ be a sequence of barriers on $N\in[\N]^{\omega}$. For every coloring $\varphi :Bl(\mathcal{B}_1,\ldots,\mathcal{B}_k)\longrightarrow X$ and sequence $\varepsilon_{i}\searrow0$, there exists $M=\{m_i : i \in \N \}\in [\N]^{\omega}$ so that for every $S=\{s_1,\ldots,s_k\},$ $T=\{t_1,t_2\ldots,t_k\}\in Bl(\mathcal{B}_1\upharpoonright_M,$ $\ldots,\mathcal{B}_k\upharpoonright_M)$ we have that
	\begin{equation*}
		d(\varphi(S),\varphi(T))<\varepsilon _{\ell},
	\end{equation*}
whenever $\min(s_1\cup t_1)=m_{\ell}$ for some $\ell\in\N$.
\end{cor}

\begin{proof} 
	Consider a sequence $\varepsilon_{i}\searrow 0$ and a coloring $\varphi:Bl(\mathcal{B}_1,\ldots,\mathcal{B}_k) \longrightarrow X$. From Theorem \ref{Teo:RamseyAnaBBarreras} we obtain a set $M_{1} \in [N]^\omega$ such that $d(\varphi(S),\varphi(T))<\varepsilon_1$ for each $S,T\in Bl(\mathcal{B}_1\upharpoonright_{M_1},$ $\ldots,\mathcal{B}_k\upharpoonright_{M_1})$. We set $m_{1}:=\min(M_1)$. Suppose that for all $i<n$ with $n\in\N$ there is $M_{i+1}\in[M_{i}/m_{i}]^{\omega}$, where $m_{i}:=\min(M_{i})$, which satisfies that $d(\varphi(S),\varphi(T))<\varepsilon_{i+1}$ for all $S,T\in Bl(\mathcal{B}_1\upharpoonright_{M_{i+1}},\ldots,\mathcal{B}_k\upharpoonright_{M_{i+1}})$. Applying Theorem \ref{Teo:RamseyAnaBBarreras} to $\varphi\upharpoonright_{Bl(\mathcal{B}_1\upharpoonright_{M_{n}/m_n},\ldots,\mathcal{B}_k\upharpoonright_{M_{n}/m_n})}$ and  $\varepsilon_{n+1}$, we get a set $M_{n+1}\in[M_n/m_n]^{\omega}$ so that $d(\varphi(S),\varphi(T))<\varepsilon_{n+1}$ for every $S,T\in Bl(\mathcal{B}_1\upharpoonright_{M_{n+1}},\ldots,\mathcal{B}_k\upharpoonright_{M_{n+1}})$ and put $m_{n+1}:=\min(M_{n+1})$. By this procedure we generate a increasing sequence $(m_i)_{i\in\N}$ of natural numbers and a sequence $(M_i)_{i\in\N}$ of infinite subsets of $N$ which satisfy the following conditions:
	\begin{enumerate}
		\item[$\bullet$] For every $i\in\N$, the number $m_i:=\min (M_i)$.
		\item[$\bullet$]  $M_1\in[N]^{\omega}$ and $M_{i+1}\in [M_{i}/m_i]^{\omega}\subseteq [M_i]^{\omega}$ for each $i \in \N$.
		\item[$\bullet$]  For every $i\in\N$, $d(\varphi(S),\varphi(T))<\varepsilon_{i}$ holds whenever $S,T\in Bl(\mathcal{B}_1\upharpoonright_{M_{i}},\ldots,\mathcal{B}_k\upharpoonright_{M_{i}})$.
	\end{enumerate}
	Now, we take $M := \{ m_i : i \in \N \}$. Notice that for $S=\{s_1,\ldots,s_k\}$, $T=\{t_1,\ldots,t_k\}\in Bl(\mathcal{B}_1\upharpoonright_{M},\ldots,\mathcal{B}_k\upharpoonright_{M})$ we get that
	\begin{equation*}
	d(\varphi(S),\varphi(T))<\varepsilon _{\ell},
	\end{equation*}
	where  $m_{\ell}=\min(s_1\cup t_1)$, because of $S,T \in Bl(\mathcal{B}_1\upharpoonright_{M_{\ell}},\ldots,\mathcal{B}_k\upharpoonright_{M_{\ell}})$.
	Hence, $M$ is the required set.
\end{proof}

The Ramsey Theorem for Analysts, presented in the article \cite[Th.2.2.1]{Schlu}, is stated with a property of convergence. In analogy to this idea, we obtain the version of this theorem for blocks of a finite sequence of barriers. To understand this procedure well, we enunciate the following notion of convergence.

\begin{defin}\label{Def:Convergencia} 
Let $(X,d)$ be a metric space, let $k \in \N$ and let $(\mathcal{B}_i)_{i=1}^k$ be a sequence of barriers on $N \in [\N]^\omega$. A coloring $\varphi:Bl(\mathcal{B}_1,\ldots,\mathcal{B}_k) \longrightarrow X$ converges to $x \in X$ if for each $\varepsilon > 0$ there is $\ell \in \N$ such that
	\begin{equation*}
		d(\varphi(S),x)<\varepsilon,
	\end{equation*}
for all $S=\{ s_1, \ldots ,s_k\} \in Bl(\mathcal{B}_1\upharpoonright_{N/\ell},\ldots,\mathcal{B}_k\upharpoonright_{N/\ell}) $. This converge is denoted by 
	\begin{equation*}
		\lim_{\substack{Bl(\mathcal{B}_1,\ldots,\mathcal{B}_k) \ni S   \longrightarrow \infty}}\varphi (S)= x.
	\end{equation*}
\end{defin}

The compactness of a metric space allows that every coloring (from a family of blocks of barriers to this metric space) has a convergent restriction.

\begin{cor}\label{Cor:ConverBBarrera} 
Let	$(K,d)$ be a compact metric space, let $k\in \N$ and let $(\mathcal{B}_i)_{i=1}^k$ be a sequence of barriers on $N\in[\N]^{\omega}$. For every coloring $\varphi: Bl(\mathcal{B}_1,\ldots,\mathcal{B}_k) \longrightarrow K$, there are $M \in [N]^\omega$ and $x \in K$ such that
	\begin{equation*}
		\lim_{\substack{Bl(\mathcal{B}_1\upharpoonright_M,\ldots,\mathcal{B}_k\upharpoonright_M)  \ni S \longrightarrow \infty}}\varphi (S) = x.
	\end{equation*}
\end{cor}

\begin{proof} 
 It follows from Corollary \ref{Cor:RamseyAnaBBarrera} and that a topological space if and only if every collection of closed sets with the finite intersection property has non-empty intersection.
\end{proof}

Corollary \ref{Cor:ConverBBarrera} is also known as the Ramsey Theorem on Blocks of Barriers for Analysts. It is easy to see that if we consider compact metric spaces, then  Theorem \ref{Teo:RamseyAnaBBarreras} and Corollary \ref{Cor:ConverBBarrera} are equivalent.

\section{Block Oscillation Stability}
	
The following notion is a slight modification of one that was introduced in the unpublished paper \cite{cg} and it is the main motivation of this paper. 	

\begin{defin}
		Let $k\in\N$ and $\varepsilon>0$. A normalized sequence $(x_i)_{i \in \N}$ in a Banach space $(X,\|\cdot\|)$ is $(k,\varepsilon)$-oscillation stable if for each $s=\{s(1),\ldots,s(k)\}$,  $t=\{t(1),\ldots,t(k)\} \in [\N]^k$ we have that
	\begin{equation*}
		\big| \|\sum_{i =1}^{k}a_i x_{s(i)} \| - \|\sum_{i=1}^{k} a_i x_{t(i)} \| \big| < \varepsilon,
	\end{equation*}
for all $(a_i)_{i=1}^{k} \in [-1,1]^{k}$.
\end{defin}

The relevance of the previous property is its relation with the Ramsey Theorem, which will be exposed and proved immediately later in this section via a generalization of the preceding notion. 

\medskip

Given $k\in\N$, observe that the barrier $[\N]^k$ determines which block linear combinations are compared in the definition of $(k,\varepsilon)-$oscillation stable sequence. This motivates a generalization of the idea of $(k,\varepsilon)-$oscillation stability of a sequence by using a family of blocks of a finite sequence of barriers as we formalize it below.

\medskip

\begin{defin}\label{def:OscBEst}
	Let $k\in\N$, let $(\mathcal{B}_i)_{i=1}^k$ be a sequences of barriers on $\N$ and let $\varepsilon>0$.
	A normalized sequence $(x_i)_{i\in\N}$ in a Banach space $(X,\|\cdot\|)$ is $((\mathcal{B}_i)^k_{i=1},\varepsilon)-${\it block oscillation stable} if for every $S=\{s_1,\ldots,s_k\}$, $T=\{t_1,\ldots,t_k\}\in {Bl}(\mathcal{B}_1,\ldots,\mathcal{B}_k)$  we have that
	\begin{equation*}
	\big | \|\sum_{i=1}^{k}a_i\mathcal{X}(s_i)\|-\|\sum_{i=1}^{k}a_i\mathcal{X}(t_i)\|\big |<\varepsilon,
	\end{equation*}
	for all $(a_i)_{i=1}^k\in[-1,1]^k$. Particularly, a $((\mathcal{B}_i)^k_{i=1},\varepsilon)-$block oscillation stable sequence is called $(\mathcal{B},k,\varepsilon)-$\textit{block oscillation stable} sequence when $\mathcal{B}=\mathcal{B}_1=\ldots=\mathcal{B}_k$.
\end{defin}

Let us remark that if the sequence $(x_i)_{i\in\N}$
 is $([\N]^1,k,\varepsilon)-$block oscillation stable, then it is $(k,\varepsilon)-$oscillation stable.

\medskip

We will see in the following theorem that every normalized sequence contains a subsequence that satisfies the condition of Definition \ref{def:OscBEst}. For this purpose, we first describe the metric spaces that were crucial in the proof of the Brunel-Sucheston Theorem (Theorem \ref{Teo:BruSu}), according to the notes of Th. Schlumprecht \cite{Schlu}. Later we will study the behavior of certain colorings in these metric spaces.

\addvspace{\smallskipamount}

For each $k \in \N$, we consider the metric space
\begin{equation*}
	\mathcal{M}_k  = \{ \rho: \mathbb{R}^k \to [0,\infty) : \rho \ \text{is a norm and}~\forall i= 1,\ldots, k~(\rho(e_i)= 1)\}
\end{equation*}
whose metric is given by
\begin{equation*}
	d_k(\rho_1,\rho_2) = \sup \left\{ \big|\rho_1(a) - \rho_2(a)\big| : a=(a_1,\ldots,a_k) \in [-1,1]^k \right\}.
\end{equation*}
for every pair of norms $\rho_1, \rho_2 \in \mathcal{M}_k$.

A false assertion in one of the proofs of the Brunel-Sucheston Theorem (Theorem \ref{Teo:BruSu}) was ``for all $k\in\N$ the metric space $(\mathcal{M}_k,d_k)$ is compact''. In fact, we show next a simple counterexample of this, which we include since it lies in the unpublished article \cite{cg}.

\begin{exam}
Let $(\mathcal{M}_2,d_2)$ be the previously defined metric space. We consider the sequence $(\|\cdot\|_{n})_{n\in\N}$ of norms over $\R^2$ defined for each $n\in\N$ by
\begin{align*}
\|a\|_n=\max\{|a_1-a_2|,\frac{1}{n}|a_2|\}
\end{align*}
for all $a=(a_1,a_2)\in\R^2$. It is not hard to see that $(\|\cdot\|_{n})_{n\in\N}$ is a Cauchy sequence. From this we deduce that $\{(\|a\|_{n})_{n\in\N}:a\in\R\}$ is uniformly Cauchy in $\R$. We define $\rho:\R^2\rightarrow[0,\infty)$ as $\rho(a)=\lim_{n \to \infty}\|a\|_{n}$ for all $a\in\R^2$. In fact, 
	\begin{align*}
		\rho(a)=\left\{\begin{array}{cl}
			|a_1-a_2|&\text{if}~a_1\neq a_2\\
			0& \text{if}~a_1=a_2
		\end{array}\right.
	\end{align*} 
for all $a\in\R^2$. By the definition, we have that $\lim\limits_{n\to\infty}\|\cdot\|_n=\rho$. However, $\rho$ is not a norm because $\rho((1,1))=0$. Hence $(\mathcal{M}_2,d_2)$ is not a compact metric space.\qed 
\end{exam}

However, we obtain compactness when we replace ``norm'' by ``seminorm'' as follows
\begin{equation*}
	\mathcal{N}_k = \{ \rho : \mathbb{R}^k \longrightarrow [0,\infty) : \rho \ \text{is a seminorm and}~\forall i= 1,\ldots,k~(\rho (e_i) = 1) \}.
\end{equation*}

We omit the proof of the following easy lemma which can be proved by using the completeness and the total boundedness of a compact metric space.

\begin{lem} 
\label{CompactoM}
For every natural number $k$, the metric space $(\mathcal{N}_{k},d_k)$ is compact.
\end{lem}

Notice that for every $k\in\N$ and $\rho\in\mathcal{N}_k $ we get that
$$\rho(a)\leq\sum_{i=1}^k\rho(a_ie_i)=\sum_{i=1}^{k}|a_i|= \leq \|a\|_{\ell_1}$$ for all $a=(a_1,\cdots, a_k)=\sum_{i=1}^ka_ie_i\in \mathbb{R}^k$.

\medskip

In the sequel, let us introduce the colorings that we will use in the next chapter:
\medskip

Let $(x_i)_{i \in \N}$  be a normalized sequence in a Banach space $(X, \|\cdot\|)$, let $k\in\N$ and let $(\mathcal{B}_i)_{i=1}^k$ be a sequence of barriers on $\N$. We define the coloring $\Psi_k:{Bl}(\mathcal{B}_1,\ldots,\mathcal{B}_k) \longrightarrow \mathcal{N}_k$, which is associated to $(x_i)_{i\in\N}$, by
\begin{equation*}
	\Psi_k(S)(\sum_{i=1}^{k}a_i e_i ):= \|\sum_{i=1}^{k}a_i \mathcal{X}(s_i)\| 
\end{equation*}
for every $S=\{s_1,\ldots,s_k\}\in{Bl}(\mathcal{B}_1,\ldots,\mathcal{B}_k)$ and $(a_i)_{i=1}^k \in [-1,1]^k$. For convenience, the sequences associated to the colorings $\Psi_k'$s will be kept implicit. In a very particular case, when $(\mathcal{B}_1,\ldots,\mathcal{B}_k)=([\N]^1,\ldots,[\N]^1)$, the coloring $\Psi_k:Bl^k([\N]^1)\longrightarrow \mathcal{N}_k$ coincides with the function $\psi_k: [\N]^k \longrightarrow \mathcal{N}_k$ given by 
\begin{align*}
\psi_k(s)(\sum_{i=1}^{k}a_i e_i ):= \|\sum_{i=1}^{k}a_ix_{s(i)}\|
\end{align*}
for each $s=\{s(1),\ldots,s(k)\}\in[\N]^k$ and $(a_i)_{i=1}^k \in [-1,1]^k$. Also, observe that for each $S\in Bl(\mathcal{B}_1,\ldots,\mathcal{B}_k)$  and $s\in[\N]^k$
\begin{align*}
&\|\sum_{i=1}^{k}a_i\mathcal{X}(s_i)\| \leq \sum_{i=1}^{k}\|a_i \mathcal{X}(s_i)\| \leq \sum_{i=1}^{k}|a_i| = \|\sum_{i=1}^{k}a_ie_i \|_{\ell_1}~\text{and}\\
&\|\sum_{i=1}^{k}a_ix_{s(1)}\|\leq \sum_{i=1}^{k}\|a_i x_{s(i)}\| \leq \sum_{i=1}^{k}|a_i| = \|\sum_{i=1}^{k}a_ie_i \|_{\ell_1},
\end{align*}
for all $(a_i)_{i=1}^k\in[-1,1]^k$

\medskip

The following lemma assures that every coloring $\Psi_k$ associated to a normalized basic sequence converges to a norm in $(\mathcal{N}_k,d_k)$ for all $k\in\N$.

\begin{lem}\label{lem:normaBk}
	Let be $(x_i)_{i \in \N}$ be a normalized basic sequence in a Banach space $(X,\|\cdot\|)$. For each
	$k\in\N$ and sequence of barriers $(\mathcal{B}_i)_{i=1}^k$ on $\N$, there are a norm  $\rho_k: \mathbb{R}^k \longrightarrow [0,\infty)$ and $M\in [\N]^\omega$ such that
	\begin{align*}
	\lim_{\substack{ {Bl}(\mathcal{B}_1\upharpoonright_M,\ldots,\mathcal{B}_k\upharpoonright_M)\ni S\longrightarrow \infty}}\Psi_k (S)= \rho_k,
	\end{align*}
this convergence is established within the compact metric space $(\mathcal{N}_k,d_k)$.
\end{lem}

\begin{proof}
	Let $C$ be the basis constant of $(x_i)_{i\in\N}$. We fix $k\in\N$ and a sequence of barriers $(\mathcal{B}_i)_{i=1}^k$ on $\N$. According to Corollary \ref{Cor:ConverBBarrera}, we get $M\in [\N]^\omega$ and $\rho_k \in\mathcal{N}_k$ so that for each $\varepsilon > 0$ there is $\ell \in M$ such that
	\begin{align}\label{ConvergenciaLema}
	d_k(\Psi_k(S),\rho_k) = \sup \{ \big|  \|\sum_{i =1}^{k}a_i\mathcal{X}(s_i) \| - \rho_k\big(\sum_{i=1}^{k} a_i e_i \big) \big| : (a_i)_{i=1}^{k} \in [-1,1]^{k} \} < \frac{\varepsilon}{C}
	\end{align}
	for all $S=\{s_1,\ldots,s_k\}\in{Bl}(\mathcal{B}_1\upharpoonright_{M/\ell},\ldots,\mathcal{B}_k\upharpoonright_{M/\ell})$. It remains to verify that $\rho_k$ is a norm. Suppose that we find $(b_i)_{i=1}^{k} \in [-1,1]^{k}\setminus\{(0,\ldots,0)\}$ such that $\rho_k\big(\sum_{i=1}^{k} b_i e_i \big) = 0$. We set $i_0:=\min\{i\leq k: b_i\neq 0\}$ and fix $0<\varepsilon<|b_{i_0}|$. From basic sequence definition and (\ref{ConvergenciaLema}), it follows that
	\begin{align*}
	|b_{i_0}|=\|\sum_{i=1}^{i_0}b_{i}\mathcal{X}(s_i)\|\leq C\|\sum_{i =1}^{k}b_i\mathcal{X}(s_i)\|\leq  Cd_k(\Psi_k(S),\rho_k) < \varepsilon,
	\end{align*}
	for all $S\in{Bl}(\mathcal{B}_1\upharpoonright_{M/\ell},\ldots,\mathcal{B}_k\upharpoonright_{M/\ell})$, which is not possible. Hence $\rho_k$  is a norm.
\end{proof}

It is intuited that the previous lemma is implicitly in the notes \cite{Schlu} when considering the sequence $([\N]^1,\ldots,[\N]^1)$ of length $k$, where $k\in\N$.

\begin{cor}
Let $(x_i)_{i \in \N}$ be a normalized basic sequence in a Banach space $(X,\|\cdot\|)$. For each $k \in \N$, there are $M\in[\N]^{\omega}$  and a norm $\rho_k: \mathbb{R}^k \longrightarrow [0,\infty)$ so that
	\begin{equation*}
		\lim_{\substack{s\in [M]^k \rightarrow \infty}}\psi_k (s)= \rho_k
	\end{equation*}
in the compact metric space $(\mathcal{N}_k,d_k)$. 
\end{cor}

We will see in the Theorem \ref{TeoPrincipalBarreras} that the Ramsey Theorem is equivalent to the following result and its corollary.

\begin{teo}\label{TeoBOscilacion}
Let $(x_i)_{i\in\N}$ be a normalized sequence in a Banach space $(X,\|\cdot\|)$. For each $k\in\N$, each sequence of barriers $(\mathcal{B}_i)_{i=1}^k$ on $\N$ and each $\varepsilon>0$, there is $M\in[\N]^{\omega}$ so that $(x_i)_{i\in M}$ is $((\mathcal{B}_i\upharpoonright_{M})_{i=1}^k,\varepsilon)-$block oscillation stable.
\end{teo}
\begin{proof}
	For each $k\in\N$, each sequence of barriers $(\mathcal{B}_i)_{i=1}^k$ on $\N$ and each $\varepsilon>0$, it follows directly from applying Theorem \ref{Teo:RamseyAnaBBarreras} to $\varepsilon$ and the coloring $\Psi_k:{Bl}(\mathcal{B}_1,\ldots,\mathcal{B}_k)\longrightarrow \mathcal{N}_k$ associated to $(x_i)_{i\in\N}$.
\end{proof}

The sequences related to the spreading models, introduced by Brunel and Sucheston, have an asymptotic property, this fact is highlighted in the article \cite{asym} with the so-called asymptotic models. For this reason, we want to emphasize that the property of block oscillation stable can be manifested asymptotically, by using blocks of a finite sequence of barriers, as we will describe below.

\begin{cor}\label{CorkBAsintoticidad}
	Let $k\in\N$ and let $(\mathcal{B}_i)^k_{i=1}$ be a sequence of barriers on $\N$. For each normalized sequence $(x_i)_{i\in\N}$ in a Banach space $(X,\|\cdot\|)$ and each $\varepsilon_i \searrow 0$ there is
	$M=\{m_i:i\in\N\}\in[\N]^{\omega}$ such that every pair $S=\{s_1,\ldots,s_k\}$, $T=\{t_1,\ldots,t_k\} \in Bl(\mathcal{B}_1\upharpoonright_M,\ldots,\mathcal{B}_k\upharpoonright_M)$ satisfies that 
		\begin{equation*}
		\big| \|\sum_{i=1}^{k}a_i \mathcal{X}(s_i) \| - \|\sum_{i=1}^{k} a_i \mathcal{X}(t_i) \| \big| < \varepsilon_{\ell},
		\end{equation*}
	where $\min(s_1\cup t_1)=m_{\ell}$, for all $(a_i)_{i=1}^{k} \in [-1,1]^{k}$.
\end{cor}
\begin{proof}
	Let $\varepsilon_i\searrow 0$.
	Inductively we will find a sequence $(M_j)_{j\in\N}$ of infinite subsets of $\N$ and a strictly increasing sequence $(m_j)_{j\in\N}$ in $\N$ that have the following properties:
	\begin{itemize}
		\item for each $j\in\N$, the subsequence $(x_{i})_{i\in M_j}$ is $((\mathcal{B}_i\upharpoonright_{M_j})_{i=1}^k,\varepsilon_j)-$block oscillation stable, and
		\item $M_1\in[\N]^{\omega}$ and $M_{j+1}\in[M_{j}/m_{j}]^{\omega}\subseteq[M_{j}]^{\omega}$ where $m_j:=\min(M_j)$ for every $j\in\N$.
	\end{itemize}
	Indeed, suppose that we have found a finite sequence $(M_j)_{j=1}^n$ of infinite subsets of $\N$ and a strictly increasing finite sequence $(m_j)_{j=1}^n$ in $\N$ that satisfy the two conditions listed above. Consider the normalized sequence $(x_i)_{i\in M_{n}/m_{n}}$ and $\varepsilon_{n+1}$. According to Theorem \ref{TeoBOscilacion}, we get
	$M_{n+1}\in[M_{n}/m_{n}]^{\omega}$ so that $(x_i)_{i\in M_{n+1}}$ is $((\mathcal{B}_i\upharpoonright_{M_{n+1}})_{i=1}^k,\varepsilon_{n+1})-$block oscillation stable, i.e., for every $S=\{s_1,\ldots,s_k\}$, $T=\{t_1,\ldots,t_k\}\in Bl(\mathcal{B}_1\upharpoonright_{M_{n+1}},\ldots,\mathcal{B}_k\upharpoonright_{M_{n+1}})$ one obtains that
	\begin{align*}
	\big|\|\sum_{i=1}^ka_ i\mathcal{X}(s_i)\|-\|\sum_{i=1}^ka_i\mathcal{X}(t_i)\|\big|<\varepsilon_{n+1},
	\end{align*}
	for all $(a_i)_{i=1}^k\in[-1,1]^k$. We propose $M:=\{m_j:j\in\N\}$ as the desired set. Indeed, we observe that if $S=\{s_1,\ldots,s_k\}$, $T=\{t_1,\ldots,t_k\}\in Bl(\mathcal{B}_1\upharpoonright_{M},\ldots,\mathcal{B}_k\upharpoonright_{M})$, then 
	\begin{align*}
	\big|\|\sum_{i=1}^ka_i\mathcal{X}(s_i)\|-\|\sum_{i=1}^ka_i\mathcal{X}(t_i)\|\big|<\varepsilon_{\ell}
	\end{align*}
	whenever $\min(s_1\cup t_1)=m_{\ell}$, for each $(a_i)_{i=1}^k\in[-1,1]^k$.
\end{proof}

Corollary \ref{CorkBAsintoticidad} motives the following block asymptotic oscillation stability property of a normalized sequence in a Banach space.

\begin{defin}
	Let $k$ be a natural number and $(\mathcal{B}_i)^k_{i=1}$ be a sequence of barriers on $\N$. A normalized sequence $(x_i)_{i \in \N}$ in a Banach space is $(\mathcal{B}_i)^k_{i=1}-$\textit{block asymptotic oscillation stable} if there exits $\varepsilon_i\searrow 0$ such that every $S=\{s_1,\ldots,s_k\}$, $T=\{t_1,\ldots,t_k\} \in Bl(\mathcal{B}_1,\ldots,\mathcal{B}_k)$ satisfy
	\begin{equation*}
	\big| \|\sum_{i=1}^{k}a_i \mathcal{X}(s_i) \| - \|\sum_{i=1}^{k} a_i \mathcal{X}(t_i) \| \big| < \varepsilon_{\min(s_1\cup t_1)}
	\end{equation*}
	for each $(a_i)_{i=1}^{k} \in [-1,1]^{k}$.
\end{defin}

The following is a condition that is equivalent to the notion of block asymptotic oscillation stability and facilitates its use in some cases.

\begin{lem}\label{OscBAsintEqui}
	Let $k\in\N$ and let $(\mathcal{B}_i)_{i=1}^k$ be a sequence of barriers on $\N$. For a normalized sequence $(x_i)_{i \in \N}$ in a Banach space $(X,\|\cdot\|)$ we have the following equivalent conditions:
	\begin{enumerate}
		\item The sequence $(x_i)_{i \in \N}$ is $(\mathcal{B}_i)^k_{i=1}-$block asymptotic oscillation stable.
		\item For each  $\varepsilon>0$ there is $n \in \N$ so that if $S=\{s_1,\ldots,s_k\}$, $T=\{t_1,\ldots, t_k\} \in Bl(\mathcal{B}_1\upharpoonright_{\N/n},\ldots,\mathcal{B}_k\upharpoonright_{\N/n} ) $, then
		\begin{equation*}
		\big| \|\sum_{i =1}^{k}a_i\mathcal{X}(s_i) \| - \|\sum_{i=1}^{k} a_i \mathcal{X}(t_i) \| \big| < \varepsilon
		\end{equation*}
		for all $(a_i)_{i=1}^{k} \in [-1,1]^{k}$.
	\end{enumerate}
\end{lem}

\begin{proof}
	$(1)\Rightarrow(2).$ By hypothesis, there exists $\varepsilon_i\searrow0$ so that for every $S=\{s_1,\ldots,s_k\}$, $T=\{t_1,\ldots,t_k\} \in Bl(\mathcal{B}_1,\ldots,\mathcal{B}_k)$ we get that
	\begin{equation*}
	\big| \|\sum_{i=1}^{k}a_i \mathcal{X}(s_i) \| - \|\sum_{i=1}^{k} a_i \mathcal{X}(t_i) \| \big| < \varepsilon_{\min (s_1\cup t_1)}
	\end{equation*}
	for each $(a_i)_{i=1}^{k} \in [-1,1]^{k}$. Now we fix $\varepsilon>0$ and choose $n\in\N$ such that $\varepsilon_n\leq \varepsilon$. Hence, for each $S$, $T \in Bl(\mathcal{B}_1\upharpoonright_{\N/n},\ldots,\mathcal{B}_k\upharpoonright_{\N/n} ) $, we have that
	\begin{equation*}
	\big| \|\sum_{i =1}^{k}a_i\mathcal{X}(s_i) \| - \|\sum_{i=1}^{k} a_i \mathcal{X}(t_i) \| \big| < \varepsilon
	\end{equation*}
	for all $(a_i)_{i=1}^{k} \in [-1,1]^{k}$.
	
	$(2)\Rightarrow(1).$ First, we notice that if $S=\{s_1,\ldots,s_k\}$, $T=\{t_1,\ldots,t_k\}  \in Bl(\mathcal{B}_1,\ldots,\mathcal{B}_k)$, then
	\begin{align*}
	\big| \|\sum_{i =1}^{k}a_i\mathcal{X}(s_i) \| - \|\sum_{i=1}^{k} a_i \mathcal{X}(t_i) \| \big|
	&\leq \sum_{i=1}^k|a_i|\|\mathcal{X}(s_i)\|+\sum_{i=1}^k|a_i|\|\mathcal{X}(t_i)\|\\
	&=\sum_{i=1}^k|a_i|+\sum_{i=1}^k|a_i|\leq 2k
	\end{align*}
	for each $(a_i)_{i=1}^{k} \in [-1,1]^{k}$. By assumption, we find $n_1\in\N$ such that for every $S$, $T \in Bl(\mathcal{B}_1\upharpoonright_{\N/n_1},\ldots,\mathcal{B}_k\upharpoonright_{\N/n_1} ) $ the inequality
	\begin{equation*}
	\big| \|\sum_{i =1}^{k}a_i\mathcal{X}(s_i) \| - \|\sum_{i=1}^{k} a_i \mathcal{X}(t_i) \| \big| <\frac{1}{2}
	\end{equation*}
	holds for all $(a_i)_{i=1}^{k} \in [-1,1]^{k}$. For each $i\leq n_1$ we define
	\begin{align*}
	\varepsilon_i:=2k+\frac{n_1+1-i}{2^{n_1}}.
	\end{align*}
	From the construction, it follows that $(\varepsilon_{i})_{i\leq n_1}$ is strictly decreasing, $\varepsilon_{i}>2k$ for all $i\leq n_1$ and $\varepsilon_{n_1}=2k+\frac{1}{2^{n_{1}}}$. 
	By applying recursively the hypothesis, for every $j\in\N$ we obtain $n_j\in\N$ with $n_j>n_{j-1}$ such that if $S$, $T\in Bl(\mathcal{B}_1\upharpoonright_{\N/n_j},\ldots,\mathcal{B}_k\upharpoonright_{\N/n_j} )$, then
	\begin{equation*}
	\big| \|\sum_{i =1}^{k}a_i\mathcal{X}(s_i) \| - \|\sum_{i=1}^{k} a_i \mathcal{X}(t_i) \| \big| < \frac{1}{2^{j}}
	\end{equation*}
	for all $(a_i)_{i=1}^{k} \in [-1,1]^{k}$. Later, taking $d_j:=\max\{j,n_j-n_{j-1}\}$ for each $j\in\N/1$, we define 
	\begin{align*}
	\varepsilon_i:=\frac{1}{2^{j-1}}+\frac{n_j+1-i}{2^{d_j}}
	\end{align*}
	for every $i\in\N$ that satisfies the restriction $n_{j-1}<i\leq n_j$. It follows that $(\varepsilon_{i})_{i\leq n_j}$ is strictly decreasing, $\varepsilon_{i}>\frac{1}{2^{j-1}}$ for each $i\leq n_j $  and $\varepsilon_{n_j}=\frac{1}{2^{j-1}}+\frac{1}{2^{d_j}}$. Moreover, $(\varepsilon_i)_{i\in\N}$ is a strictly decreasing sequence that converges to zero. Now, we choose $S$, $T\in Bl(\mathcal{B}_1,\ldots,\mathcal{B}_k)$. If $\min (s_1\cup t_1)\leq n_1$, then
	\begin{align*}
	\big| \|\sum_{i=1}^{k}a_i \mathcal{X}(s_i) \| - \|\sum_{i=1}^{k} a_i \mathcal{X}(t_i) \| \big|<2k<\varepsilon_{\min(s_1\cup t_1)}
	\end{align*}
	for each $(a_i)_{i=1}^{k} \in [-1,1]^{k}$. In the opposite case, since $(n_j)_{j\in\N}$ is strictly increasing, there exists $\ell\in \N$ such that $\ell=\max\{j\in\N:n_j<\min(s_1\cup t_1)\}$ and $\min(s_1\cup t_1)\leq n_{\ell+1}$. From this it follows that $
	S$, $T \in Bl(\mathcal{B}_1\upharpoonright_{\N/n_{\ell}},\ldots,\mathcal{B}_k\upharpoonright_{\N/n_{\ell}})$ and
	\begin{align*}
	\big| \|\sum_{i=1}^{k}a_i \mathcal{X}(s_i) \| - \|\sum_{i=1}^{k} a_i \mathcal{X}(t_i) \| \big| <\frac{1}{2^{\ell}}<\varepsilon_{\min(s_1\cup t_1)},
	\end{align*}
	for each $(a_i)_{i=1}^{k} \in [-1,1]^{k}$. Hence, $(x_i)_{i\in\N}$ is $(\mathcal{B}_i)_{i=1}^k-$block asymptotic oscillation stable.
\end{proof}

Now, we introduce the notion of block asymptotic oscillation stable sequence using an infinite sequence of barriers on $\N$, which is similar to the notion of asymptoticity studied in \cite{asym}.

\begin{defin}
	Let $(\mathcal{B}_i)_{i\in\N}$ be a sequence of barriers on $\N$. A normalized sequence $(x_i)_{i \in \N}$ in a Banach space $(X,\|\cdot\|)$ is  $(\mathcal{B}_i)_{i\in\N}-$\textit{block asymptotic oscillation stable} if there exists a sequence $\varepsilon_i\searrow 0$ such that for each $k\in\N$ and  $S=\{s_1,\ldots,s_k\}$, $T=\{t_1,\ldots,t_k\}\in Bl(\mathcal{B}_1\upharpoonright_{\N/(k-1)},\ldots,\mathcal{B}_k\upharpoonright_{\N/(k-1)})$ one gets that
	\begin{equation*}
	\big|\|\sum_{j = 1}^k a_j\mathcal{X}(s_j) \| -\|\sum_{j = 1}^k a_j\mathcal{X}(t_j) \|\big| < \varepsilon_{\min (s_1\cup t_1)},
	\end{equation*}
	for all $(a_i)_{i=1}^{k} \in [-1,1]^{k}$.
\end{defin}

The reader who knows the Brunel-Sucheston Theorem (Th. \ref{Teo:BruSu}) may imagine where we are going to, a generalization of such theorem. To achieve this goal, we need the following theorem (after Theorem \ref{TeoPrincipalBarreras} 
it will not be difficult to see that this theorem is equivalent to Ramsey's Theorem).

\begin{teo}\label{TeoBAsintoticidad}
	Let $(\mathcal{B}_i)_{i\in\N}$ be a sequence of barriers on $\N$. If $(x_i)_{i\in\N}$ is a normalized sequence in a Banach space $(X,\|\cdot\|)$, then for each $\varepsilon_i\searrow 0$ there is $M=\{m_i:i\in\N\}\in[\N]^{\omega}$ such that for every $k\in\N$ and $S=\{s_1,\ldots,s_k\}$, $T=\{t_1,\ldots,t_k\}\in Bl(\mathcal{B}_1\upharpoonright_{M/m_{k-1}},$ $\ldots,\mathcal{B}_k\upharpoonright_{M/m_{k-1}})$ one has that
	\begin{equation*}
	\big|\|\sum_{i = 1}^k a_i\mathcal{X}(s_i) \| -\|\sum_{i = 1}^k a_i\mathcal{X}(t_i) \|\big| < \varepsilon_{\ell},
	\end{equation*}
	where $\min (s_1\cup t_1)=m_{\ell}$, for each $(a_i)_{i =1}^ k \in [-1,1]^k$. In other words, the subsequence $(x_i)_{i\in M}$ is $(\mathcal{B}_i\upharpoonright_{M})_{i\in\N}-$block asymptotic oscillation stable. 
\end{teo}
\begin{proof}
	Let $(x_i)_{i\in\N}$ be a normalized sequence in a Banach space $(X,\|\cdot\|)$ and let $\varepsilon_i\searrow 0$. By Theorem \ref{TeoBOscilacion}, we get $M_1\in[\N]^{\omega}$ such that if $S=\{s_1\}$,~$T=\{t_1\}\in Bl(\mathcal{B}_1\upharpoonright_{M_1})$, then
	\begin{align*}
	\big| \|a\mathcal{X}(s_1)\|-\|a\mathcal{X}(t_1)\| \big| <\varepsilon_1
	\end{align*}
	for each $a\in[-1,1]$. We set $m_1:=\min (M_1)$.
	Recursively applying Theorem \ref{TeoBOscilacion}, we obtain a sequence $(M_j)_{j\in\N}$ of infinite subsets of $\N$ and a strictly increasing sequence $(m_j)_{j\in\N}$ in $\N$ so that:
	\begin{itemize}
		\item For every $j\in\N$ $M_{j+1}\in[M_{j}/m_{j}]^{\omega}$ where  $m_j:=\min(M_j)$.
		\item For each $j\in\N$ if $S=\{s_1,\ldots,s_j\}$, $T=\{t_1,\ldots,t_j\}\in Bl(\mathcal{B}_1\upharpoonright_{M_j},\ldots,\mathcal{B}_j\upharpoonright_{M_j})$, then we have 
		\begin{align*}
		\big| \|\sum_{i=1}^{j}a_i \mathcal{X}(s_i) \| - \|\sum_{i=1}^{j} a_i \mathcal{X}(t_i) \| \big| <\varepsilon_j,
		\end{align*}
		for each $(a_i)_{i=1}^{j} \in [-1,1]^{j}$.
	\end{itemize}
	We consider the set $M:=\{m_j:j\in\N\}$. Fix $k\in\N$ and $S=\{s_1,\ldots,s_k\}$, $T=\{t_1,\ldots,t_k\}\in Bl(\mathcal{B}_1\upharpoonright_{M/m_{k-1}},\ldots,\mathcal{B}_k\upharpoonright_{M/m_{k-1}})$. Let $\ell\in\N$ such that $m_{\ell}=\min(s_1\cup t_1)$. According property \ref{Bsegmento} of barriers, if it is necessary, we choose $S'=\{s'_i\in\mathcal{B}_i\upharpoonright_{M}:k+1\leq i\leq\ell\}$ and $T'=\{t'_i\in\mathcal{B}_i\upharpoonright_{M}:k+1\leq i\leq\ell\}$ such that $S\cup S',T\cup T'\in {Bl}(\mathcal{B}_1\upharpoonright_{M/m_{\ell-1}},\ldots,\mathcal{B}_{\ell}\upharpoonright_{M/m_{\ell-1}})$. 
	From this, it follows that
	\begin{align*}
	\big|\|\sum_{i = 1}^k a_i \mathcal{X}(s_i) \| -\|\sum_{i= 1}^k a_i \mathcal{X}(t_i)\| \big|
	=&\big|\|\sum_{i = 1}^{k} a_i \mathcal{X}(s_i)+ \sum_{i = k+1}^{\ell} 0 \mathcal{X}(s'_i) \|\\
	&-\|\sum_{i = 1}^{k} a_i \mathcal{X}(t_i)+ \sum_{i = k+1}^{\ell} 0 \mathcal{X}(t'_i)  \|\big|<\varepsilon_{\ell},
	\end{align*}		
	for each $(a_i)_{i =1}^k \in [-1,1]^k$.
	Thus $M$ is the required set.
\end{proof}

It is worth mentioning that the previous result is important because it ensures the existence of sequences with the property of block asymptotic oscillation stability concerning an infinite sequence of barriers. 

\medskip

As a particular case of Theorem \ref{TeoBAsintoticidad}, when we consider the sequence $([\N]^1,\ldots,[\N]^1,\ldots)$, we obtain the following known result.

\begin{cor}
	If $(x_i)_{i\in\N}$ is a normalized sequence in a Banach space $(X,\|\cdot\|)$, then for each $\varepsilon_i\searrow 0$ there is a subsequence $(x_{n_i})_{i\in\N}$ such that every $s$, $t\in FIN^{*}$ with $s(1)\geq|s|=k=|t|\leq t(1)$ satisfies
	\begin{equation*}
	\big|\|\sum_{j = 1}^k a_jx_{n_{s(j)}} \| -\|\sum_{j = 1}^k a_jx_{n_{t(j)}} \|\big| < \varepsilon_{\min\{s(1),t(1)\}},
	\end{equation*}
	for all $(a_i)_{i =1}^ k \in [-1,1]^k$.
\end{cor}


One of our main contributions in this paper is the following theorem that lists conditions that are equivalent to the Ramsey Theorem. It must be mentioned that the particular case of this theorem for the sequence $([\N]^1,\ldots,[\N]^1)$ of length $k$ and the sequence $([\N]^1,\ldots,[\N]^1,\ldots)$ can be found in the article \cite{cg}.

\begin{teo}\label{TeoPrincipalBarreras}
	The following assertions are equivalent:
	\begin{enumerate}[label={(\arabic*)}]
		\item  The Ramsey Theorem (Th.\ref{Teo:Ramsey}).
		\item  The Ramsey Theorem on Barriers (Th. \ref{Teo:BarrerasRamseyNoNula}).
		\item  The Ramsey Theorem on Blocks of Barriers (Th. \ref{Teo:RamseyBBarreras}).
		\item The Ramsey Theorem on Blocks of Barriers for Analysts (Th. \ref{Teo:RamseyAnaBBarreras}).
		\item  Let $(x_i)_{i\in\N}$ be a normalized sequence in a Banach space $(X,\|\cdot\|)$ and let $k\in\N/1$. For each sequence of barriers $(\mathcal{B}_i)_{i=1}^k$ on $\N$ and $\varepsilon>0$, there exists $M\in[\N]^{\omega}$ so that $(x_i)_{i\in M}$ is $((\mathcal{B}_i\upharpoonright_{M})_{i=1}^k,\varepsilon)-$block oscillation stable.
		\item Let $(x_i)_{i\in\N}$ be a normalized sequence in a Banach space $(X,\|\cdot\|)$ and let $k\in\N/1$. For each sequence of barriers $(\mathcal{B}_i)_{i=1}^k$ on $\N$, there is $M\in[\N]^{\omega}$ so that $(x_i)_{i\in M}$ is $(\mathcal{B}_i\upharpoonright_{M})_{i=1}^k-$block asymptotic oscillation stable.
		\item Let $(x_i)_{i\in\N}$ be a normalized sequence in a Banach space $(X,\|\cdot\|)$. For each sequence of barriers $(\mathcal{B}_i)_{i\in\N}$ on $\N$, there exist $M\in[\N]^{\omega}$ so that $(x_i)_{i\in M}$ is $(\mathcal{B}_i\upharpoonright_{M})_{i\in \N}-$block asymptotic oscillation stable.
	\end{enumerate}
\end{teo}  
\begin{proof}
	$(1)\Rightarrow (2).$ We proceed by transfinite induction on the lexicographical rank of the barriers. Assume that $\mathcal{B}$ is a barrier on $N\in[\N]^{\omega}$ of lexicographical rank equal to $\omega^k$ for some $k\in\N$ (we know this fact because of Corollary \ref{Cor:rankBar}). We take an arbitrary coloring $\varphi:\mathcal{B}\longrightarrow \{1,\ldots,q\}$ where $q\in\N$. By Corollary \ref{Cor:rank=omega} and Theorem \ref{Teo:BarOmk}, there is $n\in \N\cup\{0\}$ such that $[N/n]^k\subseteq\mathcal{B}$. We consider the restriction $ \varphi\upharpoonright_{\mathcal{B}\upharpoonright_{N/n}}:\mathcal{B}\upharpoonright_{N/n}\rightarrow\{1,\ldots,q\}$, where $\mathcal{B}\upharpoonright_{N/n}=[N/n]^k$. From assumption (Th. \ref{Teo:Ramsey}) we get $i\leq q$ and $M\in[N/n]^{\omega}$ so that $[M]^k\subseteq (\varphi|_{\mathcal{B}\upharpoonright_{N/n}})^{-1}(i)$. This  implies that $\mathcal{B}\upharpoonright_{M}\subseteq\varphi^{-1}(i)$.
	
	Now, suppose that for every barrier $\mathcal{B}$ with $rank(\mathcal{B})<\alpha$ (where $\alpha\geq \omega^{\omega}$) and finite coloring $\varphi:\mathcal{B}\longrightarrow\{1,\ldots,q\}$, there are $i\leq q$ and $M\in[\N]^{\omega}$ such that $\mathcal{B}\upharpoonright_{M}\subseteq\varphi^{-1}(i)$.
	Let $\mathcal{B}$ be a barrier on $N\in[\N]^{\omega}$ of lexicographical rank equal to  $\alpha$. We take an arbitrary finite coloring $\varphi:\mathcal{B}\rightarrow\{1,\ldots,q\}$. We assume without loss of generality that $\mathcal{B}_{\{n\}}\neq\emptyset$ for all $n\in N$. 
	For each $n\in N$ we define
	\begin{equation*}
	\varphi_{n}:\mathcal{B}_{\{n\}}\rightarrow \{1,\ldots,q\}~\text{as}~\varphi_{n}(s)=\varphi(\{n\}^{\frown}s)~\text{for all}~
	s\in\mathcal{B}_{\{n\}}.
	\end{equation*}
	According Lemma \ref{Lem:rankBn}, we know  that $rank(\mathcal{B}_{\{n\}}\upharpoonright M)<\alpha$ for every $n\in N$ and $M\in[N]^{\omega}$. We set $m_1:=\min (N)$. Applying the inductive hypothesis to $\varphi_{m_1}$ we get $i_1\leq q$ and $M_1\in[N/m_1]^{\omega}$ such that  $\varphi_{m_1}(\mathcal{B}_{\{m_1\}}\upharpoonright_{M_1})=\{i_1\}$. We set $m_2:=\min (M_1)$ and consider the restriction $\varphi_{m_2}\upharpoonright_{\mathcal{B}_{\{m_2\}}\upharpoonright_{M_1}}$. Again, from the inductive hypothesis we find $i_2\leq q$ and $M_2\in[M_1/m_2]^{\omega}$ such that $\varphi_{m_2}(\mathcal{B}_{\{m_2\}}\upharpoonright_{M_2})=\{i_2\}$. Thus we recursively obtain $i_j\leq q$ and $M_j\in[M_{j-1}/m_j]^{\omega}$ so that $\varphi_{m_j}(\mathcal{B}_{\{m_j\}}\upharpoonright_{M_j})=\{i_j\}$, where $m_j:=\min (M_{j-1})$, for each $j\in\N$. Later, we define a finite coloring $\phi:[\N]^1\rightarrow \{1,\ldots,q\}$ by the rule
	\begin{equation*}
	\phi(\{j\})=i_j~\text{for each}~\{j\}\in[\N]^1.
	\end{equation*}
	By  Corollary \ref{Cor:rank=omega}, we know that $rank([\N]^1)=\omega$. Hence we get $i\leq q$ and $L\in[\N]^{\omega}$ such that $\phi([L]^1)=\{i\}$. We set $M=\{m_j:j\in L\}$. Given an arbitrary element $s\in\mathcal{B}\upharpoonright M$ we have that $s=\{m_{j_1},\ldots,m_{j_{\ell_s}}\}$ with $|s|=\ell_s$. Recall that $m_{j_2}=\min (M_{j_2-1})$ and $M_{j_2-1}\subseteq M_{j_1}$ since $j_1\leq j_2-1$. From this follows that $\{m_{j_2},\ldots,m_{j_{\ell_s}}\}$ is in $\mathcal{B}_{\{m_{j_1}\}}\upharpoonright_{M_{j_1}}$. Hence, $\varphi(s)=\varphi_{m_{j_1}}(\{m_{j_2},\ldots,m_{j_{\ell_s}}\})=i$. Since $s\in\mathcal{B}\upharpoonright M$ was taking arbitrarily, we conclude that $\mathcal{B}\upharpoonright M\subseteq\varphi^{-1}(i)$.
	
	$(2)\Rightarrow (3).$ See the proof of Theorem \ref{Teo:RamseyBBarreras}.
	
	$(3)\Rightarrow(4).$ See the proof of Theorem \ref{Teo:RamseyAnaBBarreras}.
	
	$(4)\Rightarrow(5).$ See the proof of Theorem \ref{TeoBOscilacion}.
	
	$(5)\Rightarrow(7).$ See the proof of Theorem \ref{TeoBAsintoticidad}
	
	$(7)\Rightarrow(6).$ Let $(B_i)_{i=1}^k$ be a sequence of barriers on $\N$, where $k\in\N/1$. We consider the sequence $(\mathcal{B}'_i)_{i\in\N}$ where $\mathcal{B}'_i=\mathcal{B}_i$ for each $i\leq k$ and $\mathcal{B}'_i=[\N]^1$ for all $i>k$. According to assumption, there are $\varepsilon_i\searrow0$ and $M'=\{m'_i:i\in\N\}$ so that for every $n\in\N$ and
	$S=\{s_1,\ldots,s_n\}$, $T=\{t_1,\ldots,t_n\}\in Bl(\mathcal{B}'_1\upharpoonright_{M'/m'_{n-1}},\ldots,\mathcal{B}'_n\upharpoonright_{M'/m'_{n-1}})$ and $(a_i)_{i =1}^n \in [-1,1]^n$ we get that
	\begin{equation*}
	\big|\|\sum_{i = 1}^n a_i\mathcal{X}(s_i) \| -\|\sum_{i = 1}^n a_i\mathcal{X}(t_i) \|\big| < \varepsilon_{\ell}
	\end{equation*}
	where $\min (s_1\cup t_1)=m'_{\ell}$. Hence, if we set $M:=M'/m'_{k-1}$, then $(x_i)_{i\in M}$ is $(\mathcal{B}_i\upharpoonright_{M})_{i=1}^k-$block asymptotic oscillation stable.
	
	$(6)\Rightarrow(1).$  To establish this implication we modify  slightly the proof of the particular case that is given in \cite{cg}. It is enough to prove that each $k\in\N$, every coloring $\varphi:[\N]^k\longrightarrow\{1,2\}$ has a infinite monochromatic set. Let $k\in\N$ and let $\varphi:[\N]^k\longrightarrow\{1,2\}$ be a coloring.
	We consider the completion of the norm linear space $(c_{00},\|\cdot\|)$, whose norm is defined by
	\begin{equation*}
	\|x\|=\sup\Big(\big\{|a_i|:i\in\N\big\}\cup\big\{|\sum_{i \in s} a_i| : s \in [\N]^{k}~\text{and}~\varphi(s)=1\big \}\Big)
	\end{equation*}
	for each  $x=(a_i)_{i\in\N}=\sum_{i=1}^{\infty}a_ie_i\in c_{00}$.
	Notice $(e_i)_{i\in\N}$ is a normalized sequence under $\|\cdot\|$. Also, it is easy to see that if $s\in[\N]^k$ and $\varphi(s)=1$, then $\|\sum_{i \in s} e_i\|=k$. Moreover, we claim that $ \|\sum_{i \in t} e_i\|\leq k-1$ for all $t\in[\N]^k$ such that $\varphi(t) =2$. Indeed, suppose that $t\in[\N]^k$ and $\varphi(t) =2$. The vector $\sum_{i\in t}e_i$ may be written as  $(a_i)_{i\in\N}$ where $a_i=\chi_{t}(i)$ for each $i\in\N$. We consider $u \in [\mathbb{N}]^{k}$ such that $\varphi(u) = 1$. Since $t \not=u$ and $|t| = |u|$, there is  $i_0 \in t\backslash u$ and so $|\sum_{i\in u}a_i|\leq k-1$. Hence 
	\begin{align}\label{normat=1}\tag{$\bigstar$}
	\|\sum_{i \in t} e_i\| \leq k-1~\text{whenever}~ t\in[\N]^k~\text{and}~\varphi(t) =1.
	\end{align}
	Applying the hypothesis to $(e_i)_{i\in\N}$ and the sequence $([\N]^1,\ldots,[\N]^1)$ of length $k$, we find $\varepsilon_i\searrow0$ and $M'\in[\N]^{\omega}$ so that every 
	$S=\{s_1,\ldots,s_k\}$, $T=\{t_1,\ldots,t_k\}\in Bl^{k}([M']^1)$ and $(a_i)_{i =1}^k\in [-1,1]^k$ satisfy
	\begin{equation*}
	\big|\|\sum_{i = 1}^k a_i\mathcal{E}(s_i) \|-\|\sum_{i = 1}^ka_i\mathcal{E}(t_i) \|\big| < \varepsilon_{\ell}
	\end{equation*}
	where $\min (s_1\cup t_1)=m'_{\ell}$. We choose $\ell\in\N$ such that $\varepsilon_{\ell}\leq\frac{1}{2}$, and we set $M:=M'/{m'_{\ell-1}}$. Thus, for each $S=\{s_1,\ldots,s_k\}$, $T=\{t_1,\ldots,t_k\}\in Bl^{k}([M]^1)$ we have 
	\begin{equation*}
	\big| \|\sum_{i\in\cup_{j\leq k} t_j}e_i\|-\|\sum_{i\in\cup_{j\leq k} t_j}e_i\| \big|=\big|\|\sum_{i = 1}^k\mathcal{E}(s_i) \| -\|\sum_{i = 1}^k\mathcal{E}(t_i) \|\big| < \varepsilon_{\ell}\leq\frac{1}{2}.
	\end{equation*}
	 From the equality $\bigoplus_{i\leq k}[M]^1=[M]^k$ and the previous inequality follows that
	\begin{align*}
	\big| \|\sum_{i\in s}e_i\|-\|\sum_{i\in t}e_i\| \big| < \frac{1}{2},
	\end{align*}
	for every $s$, $t\in[M]^k$. Assume that there are $s$, $t \in [M]^k$ such that $\varphi(s)=1$ and $\varphi(t)=2$. Then
	\begin{align*}
	\big| \|\sum_{i \in s} e_i\|_{W} - \|\sum_{i \in t} e_i\|_{W}  \big| &< \frac{1}{2}\\
	k - \|\sum_{i \in t} e_i\|_{W} &<\frac{1}{2}\\
	k - \frac{1}{2} &<\|\sum_{i \in t} e_i\|_{W},
	\end{align*}
	which contradicts (\ref{normat=1}). Hence $M$ is the required set.
\end{proof}

\section{Block Asymptotic Models}

From what we have seen so far, the idea of generalizing spreading models naturally arises. This naturalness is observed in the use of blocks of a sequence of barriers from which the property of block oscillation stable is obtained. The main objective of this section is to generalize the notion of spreading models via blocks of sequences of barriers. In effect, we will see how these asymptotic models are obtained. To have this done, we begin by formalizing this notion.

\begin{defin}\label{BModeloAsintotico}
	Let $(x_i)_{i \in \N}$ be a normalized basic sequence in a Banach space $(X,\|\cdot\|_X)$ and let $(\mathcal{B}_i)_{i\in\N}$ be a sequence of barriers on $\N$. A Schauder basis $(y_i)_{i\in\N}$ for a Banach Space $(Y,\|\cdot\|_{Y})$ is a  {\it $(\mathcal{B}_i)_{i\in\N}-$block asymptotic model} of $(x_i)_{i \in \N}$ if there exists $\varepsilon_i\searrow 0$ such that for every $k\in\N$ and $S = \{s_1,\ldots, s_k\} \in {Bl}(\mathcal{B}_1\upharpoonright_{\N/(k-1)},\ldots,\mathcal{B}_k\upharpoonright_{\N/(k-1)})$ we have 
	\begin{equation*}
	\big|\|\sum_{j = 1}^k a_j\mathcal{X}(s_j)\|_X -\|\sum_{j = 1}^k a_jy_j \|_Y \big| < \varepsilon_{\min (s_1)},
	\end{equation*}
	for each $(a_i)_{i =1}^ k \in [-1,1]^k$. In particular, if $\mathcal{B}=\mathcal{B}_{i}$ for all $i\in\N$, then we say that $(x_i)_{i \in \mathbb{N}}$ generates $(y_i)_{i \in \mathbb{N}}$ (or $Y$) as {\it $\mathcal{B}-$block asymptotic model}.	
\end{defin}

We remark that $([\N]^1,\ldots,[\N]^1,\ldots)-$block asymptotic model coincides with the spreading model.

\medskip

Similar to spreading models, if $(y_i)_{i\in\N}$ is a $(\mathcal{B}_i)_{i\in\N}-$block asymptotic model of some basic sequence $(x_i)_{i\in\N}$, then the norm $\|\cdot\|_Y$ can be obtained as follows
\begin{align}\label{normaBklim}
\|\sum_{i=1}^ka_iy_i\|_Y=\lim\limits_{Bl(\mathcal{B}_1,\ldots,\mathcal{B}_k)\ni S\to\infty}\|\sum_{i=1}^ka_i\mathcal{X}(s_i)\|_X,
\end{align}
for all $k\in\N$ and $(a_i)_{i=1}^k\in[-1,1]^k$. 

\medskip

For every barrier $\mathcal{B}$, we will see in the next theorem that a $\mathcal{B}-$block asymptotic model is a spreading sequence. We recall that a sequence $(x_i)_{i\in\N}$ in a Banach space $(X,\|\cdot\|)$ is called {\it spreading sequence} if for each $k\in\N$ and $s=\{s(1),\ldots,s(k)\}\in[\N]^k$ we get that $$\|\sum_{i=1}^ka_ix_i\|=\|\sum_{i=1}^ka_ix_{s(i)}\|$$ for all $(a_i)_{i=1}^k\in[-1,1]^k$.

\begin{teo}\label{BModeloAsinDisperso}
Let $\mathcal{B}$ be a barrier on $\N$. If $(y_i)_{i\in\N}$ is a $\mathcal{B}-$block asymptotic model of some normalized basic sequence, then for any $k\in\N$ and $s=\{s(1),\ldots,s(k)\}\in[\N]^k$ we have that
	\begin{equation*}
	\|\sum_{i=1}^k a_iy_i\|_Y=\|\sum_{i=1}^k a_iy_{s(i)}\|_Y
	\end{equation*}
for all $(a_i)_{i=1}^k\in[-1,1]^k$.
\end{teo}

\begin{proof} 
	Suppose that $(y_i)_{i\in\N}$ is a $\mathcal{B}-$block asymptotic model of a normalized basic sequence $(x_i)_{i\in\N}$ of a Banach space $(X,\|\cdot\|)$. First, we fix $k\in\N$,  $s=\{s(1),\ldots,s(k)\}\in[\N]^k$ and $(a_i)_{i=1}^k\in[-1,1]^k$. Later, we choose  $(b_i)_{i=1}^{s(k)}\in[-1,1]^{s(k)}$ such that for each $i\leq s(k)$ 
	\begin{align*}
		b_i=\left\{
		\begin{array}{l l}
		a_j& \text{if}~i=s(j)~\text{for some}~j\leq k\\
		0& \text{otherwise}
		\end{array}
		\right..
	\end{align*}
	By the choice of $(b_i)_{i=1}^{s(k)}$, we obtain that
	\begin{align*}
		\|\sum_{i=1}^{s(k)}b_iy_i\|_Y=\|\sum_{i=1}^{k}a_iy_{s(i)}\|_Y\quad \text{and} \quad \|\sum_{i=1}^{s(k)}b_i\mathcal{X}(t_i)\|_X=\|\sum_{i=1}^{k}a_i\mathcal{X}(t_{s(i)})\|_X
	\end{align*}
	for every $T=\{t_1,\ldots,t_{s(k)}\}\in{Bl}^{s(k)}(\mathcal{B})$. From the preceding equalities and (\ref{normaBklim}), it follows that
	\begin{align*}
		\lim_{\substack{ {Bl}^k(\mathcal{B})\ni T=\{t_1,\ldots,t_k\}\rightarrow \infty}}\|\sum_{i=1}^ka_i\mathcal{X}(t_i)\|_X=&\|\sum_{i=1}^ka_iy_i\|_Y\quad\text{and}\\
		\lim_{\substack{ {Bl}^{s(k)}(\mathcal{B})\ni T=\{t_1,\ldots,t_{s(k)}\}\rightarrow \infty}}\|\sum_{i=1}^{k}a_i\mathcal{X}(t_{s(i)})\|_X=&\|\sum_{i=1}^{k}a_iy_{s(i)}\|_Y.
	\end{align*}
	Hence, given $\varepsilon>0$, there is $\ell\in\N$ such that
	\begin{equation*}
		\big| \|\sum_{i=1}^k a_i\mathcal{X}(t_i)\|_X-\big| \|\sum_{i=1}^k a_iy_i\|_Y\big|<\varepsilon.
	\end{equation*}
	for each $T=\{t_1,\ldots,t_k\}\in {Bl}^k(\mathcal{B}\upharpoonright_{\N/\ell})$. Since $\{t_{s(1)},\ldots,t_{s(k)}\}\in{Bl}^k(\mathcal{B}\upharpoonright_{\N/\ell})$ whenever $T=\{t_{1},\ldots,t_{s(k)}\}\in{Bl}^{s(k)}(\mathcal{B}\upharpoonright_{\N/\ell})$, we get 
	\begin{align*}
	\big | \|\sum_{i=1}^{k}a_i\mathcal{X}(t_{s(i)})\|_X-\|\sum_{i=1}^{k}a_iy_{i}\|_Y\big |<\varepsilon
	\end{align*}
	 for every $T\in{Bl}^{s(k)}(\mathcal{B}\upharpoonright_{\N/\ell})$. Thus
	\begin{equation*}
		\lim_{\substack{ {Bl}^{s(k)}(\mathcal{B})\ni T\rightarrow \infty}}\|\sum_{i=1}^{k}a_i\mathcal{X}(t_{s(i)})\|_X=\|\sum_{i=1}^{k}a_iy_{i}\|_Y
	\end{equation*}
	and the uniqueness of the limit guarantees that $\|\sum_{i=1}^{k}a_iy_{s(i)}\|_Y=\|\sum_{i=1}^{k}a_iy_{i}\|_Y$. Finally, since $k\in\N$, $s\in[\N]^k$ and $(a_i)_{i=1}^k\in[-1,1]^k$ were taking arbitrarily, we conclude that each $k\in\N$ and $s\in[\N]^k$ satisfy 
	\begin{align*}
	\|\sum_{i=1}^{k}a_iy_i\|_Y=\|\sum_{i=1}^{k}a_iy_{s(i)}\|_Y
	\end{align*}
	for all $(a_i)_{i=1}^k\in[-1,1]^k$.
\end{proof}

The above testifies that the concept of $\mathcal{B}-$block asymptotic model generalizes the notion of a spreading model. For this reason, it is convenient to simply name any $\mathcal{B}-$block asymptotic model as {\it $\mathcal{B}-$spreading model}. In this context, spreading model $=[\N]^1-$spreading model.

\medskip

In the following theorem, we will show how the Ramsey Theorem (via statement (7) of Theorem \ref{TeoPrincipalBarreras}) is applied in Functional Analysis to provide block asymptotic models of a normalized basic sequence. This result is a generalization of the Brunel-Sucheston Theorem (Th. \ref{Teo:BruSu}). Moreover, this theorem is equivalent to each condition of Theorem \ref{TeoPrincipalBarreras}. 

\begin{teo}\label{BrunelSuchestonBBarreras}
	For every normalized basic sequence $(x_i)_{i\in\N}$ in a Banach s pace $(X,\|\cdot\|)$ and every sequence $(\mathcal{B}_i)_{i\in\N}$ of barriers on $\N$, there exists $M\in[\N]^{\omega}$ such that the subsequence $(x_i)_{i\in M}$ induces a norm $|||\cdot|||$ on the linear space $c_{00}$ in which $(e_i)_{i\in\N}$ is a  $(\mathcal{B}_i\upharpoonright_{M})_{i\in\N}$-block asymptotic model of $(x_i)_{i\in M}$.
\end{teo}
\begin{proof}
	Let $(x_i)_{i\in\N}$ be a normalized basic sequence in a Banach space $(X,\|\cdot\|)$ and let $(\mathcal{B}_i)_{i\in\N}$ be a sequence of barriers on $\N$. Clause (7) of Theorem \ref{TeoPrincipalBarreras} yields $\varepsilon_i\searrow0$ and $M=\{m_i:i\in\N\}\in[\N]^{\omega}$ so that
	\begin{align}\label{BaBSdes1}
	\big|\|\sum_{j = 1}^k a_j\mathcal{X}(s_j) \| -\|\sum_{j = 1}^k a_j\mathcal{X}(t_j) \|\big| < \frac{\varepsilon_{\ell}}{2},
	\end{align}
	whenever $\min(s_1\cup t_1)=m_{\ell}$, holds for each $k\in\N$, $S=\{s_1\ldots,s_k\},~T=\{t_1\ldots,t_k\}\in {Bl}(\mathcal{B}_1\upharpoonright_{M/m_{k-1}},\ldots,\mathcal{B}_k\upharpoonright_{M/m_{k-1}})$ and $(a_i)_{i =1}^ k \in [-1,1]^k$. We fix $k\in \N$. It is not difficult to show that ${Bl}(\mathcal{B}_1\upharpoonright_{M},\ldots,\mathcal{B}_k\upharpoonright_{M})$ with the relation $\leq$, which is defined by 
	\begin{align*}
	&S<T~\text{if only if}~\max (s_1)<\min (t_1)\quad\text{and}\\
	&S=T~\text{if $\cup_{i\leq k}s_i=\cup_{i\leq k}t_i$}
	\end{align*}
	for any $S=\{s_1\ldots,s_k\},~T=\{t_1\ldots,t_k\}\in {Bl}(\mathcal{B}_1\upharpoonright_{M},\ldots,\mathcal{B}_k\upharpoonright_{M})$, is a directed set. So we consider the net $\big(\Psi_k(S)\big)_{S\in{Bl}(\mathcal{B}_1\upharpoonright_{M},\ldots,\mathcal{B}_k\upharpoonright_{M})}$ in the compact metric space $(\mathcal{N}_k,d_k)$. Recall that for each $S\in {Bl}(\mathcal{B}_1,\ldots,\mathcal{B}_k)$, we have
	\begin{equation*}
	\Psi_k(S)(\sum_{i=1}^ka_ie_i)=\|\sum_{i=1}^ka_i\mathcal{X}(s_i)\|~\text{for all}~(a_i)_{i=1}^k\in[-1,1]^k.
	\end{equation*}
	Given $\varepsilon>0$, we choose $\ell\in\N$ such that $\varepsilon_{\ell}\leq 2\varepsilon$ and $\ell\geq k$. Since property \ref{Bsegmento} of barriers allows us to find $R=\{r_1,\ldots,r_k\}\in{Bl}(\mathcal{B}_1\upharpoonright_{M},\ldots,\mathcal{B}_k\upharpoonright_{M})$ such that $\min (r_1)=m_{\ell}$, and every pair $S,~T\in{Bl}(\mathcal{B}_1\upharpoonright_{M},\ldots,\mathcal{B}_k\upharpoonright_{M})$ satisfies
	\begin{align*}
	S\geq T&\Leftrightarrow \min (s_1)> \max (t_1)~\text{or}~\cup_{i\leq k}s_i=\cup_{i\leq k}t_i\notag \\
	&\Leftrightarrow S\in {Bl}(\mathcal{B}_1\upharpoonright_{M/\max (t_1)},\ldots,\mathcal{B}_k\upharpoonright_{M/\max (t_1)})\cup\{T\},
	\end{align*}
	it follows that
	\begin{align*}
	d_k\big(\Psi_k(S),\Psi_k(T)\big)<\frac{\varepsilon_{\ell'}}{2}
	\leq\frac{\varepsilon_{\ell}}{2}
	\leq  \varepsilon,
	\end{align*}
	where $\min(s_1\cup t_1)=m_{\ell'} $, for all $S,T\in{Bl}(\mathcal{B}_1\upharpoonright_{M},\ldots,\mathcal{B}_k\upharpoonright_{M})$ with $S,T\geq R$. This means that $\left(\Psi_k(S)\right)_{S\in{Bl}(\mathcal{B}_1\upharpoonright_{M},\ldots,\mathcal{B}_k\upharpoonright_{M})}$  is a Cauchy-net in the compact metric space $(\mathbb{N}_k,d_k)$. Hence, there is $\rho_k\in\mathcal{N}_k$ so that $\big(\Psi_k(S)\big)_{S\in{Bl}(\mathcal{B}_1\upharpoonright_{M},\ldots,\mathcal{B}_k\upharpoonright_{M})}$ converges to $\rho_k$. We claim that
	\begin{equation}\label{BaBkconvergencia}
	\lim\limits_{{Bl}(\mathcal{B}_1\upharpoonright_{M},\ldots,\mathcal{B}_k\upharpoonright_{M})\ni S\rightarrow\infty}\Psi_k(S)=\rho_k.
	\end{equation}
	Indeed, by convergence of the net, given $\varepsilon>0$ there exists $T\in{Bl}(\mathcal{B}_1\upharpoonright_{M},\ldots,\mathcal{B}_k\upharpoonright_{M})$ such that $d_k(\Psi_k(S),\rho_k)<\varepsilon$ whenever $S\in{Bl}(\mathcal{B}_1\upharpoonright_{M},\ldots,\mathcal{B}_k\upharpoonright_{M})$ and $S\geq T$, which implies that
	\begin{equation*}
	d_k(\Psi_k(S),\rho_k)<\varepsilon
	\end{equation*}
	for every $S\in{Bl}(\mathcal{B}_1\upharpoonright_{M/\max (t_{1})},\ldots,\mathcal{B}_k\upharpoonright_{M/\max (t_{1})})$. Besides Lemma \ref{lem:normaBk} assures that $\rho_k$ is a norm on $\R^k$.
	
	\medskip
	
	Now, we are going to build a norm from the norms $\rho_k$'s that we obtained earlier. To do this, we consider the linear space $c_{00}$ and its canonical basis $(e_i)_{i\in\N}$. We define the function $|||\cdot|||:c_{00}\rightarrow [0,\infty)$ as
	\begin{equation*}
	|||\sum_{i=1}^k a_i e_i|||=\rho_k\big(\sum_{i=1}^k a_i e_{i}\big),
	\end{equation*}
	for each $k\in\N$ and $\sum_{i=1}^k a_i e_{i}\in c_{00}$. Notice that $|||\cdot|||$ is well-defined since for every $k\in\N$ and $S=\{s_1,\ldots,s_{k+1}\}\in{Bl}(\mathcal{B}_1\upharpoonright_{M},\ldots,\mathcal{B}_{k+1}\upharpoonright_{M})$ one has that
	\begin{align*}
	\Psi_{k+1}(S)\left(\sum_{i=1}^k a_i e_i+0e_{k+1}\right)=\Psi_k(S\setminus\{s_{k+1}\})\left(\sum_{i=1}^k a_i e_i\right),
	\end{align*}
	for all $(a_i)_{i=1}^k\in[-1,1]^k$. Hence, for each $k\in\N$ we get that
	\begin{align*}
	\rho_{k+1}\left(\sum_{i=1}^k a_i e_i+0e_{k+1}\right)&=\lim\limits_{{Bl}(\mathcal{B}_1\upharpoonright_{M},\ldots,\mathcal{B}_{k+1}\upharpoonright_{M})\ni S\rightarrow\infty}\Psi_{k+1}(S)\left(\sum_{i=1}^k a_i e_i+0e_{k+1}\right)\\
	&=\lim\limits_{{Bl}(\mathcal{B}_1\upharpoonright_{M},\ldots,\mathcal{B}_{k}\upharpoonright_{M})\ni T\rightarrow\infty}\Psi_{k}(T)\left(\sum_{i=1}^k a_i e_i\right)\\
	&=\rho_k\left(\sum_{i=1}^k a_i e_i\right)
	\end{align*}
	for all $(a_i)_{i=1}^k\in[-1,1]^k $. Furthermore, $|||\cdot|||$ is a norm on $c_{00}$ because, for every $k\in\N$, $\rho_k$ is a norm on $\R^k$. We denote the completion of the normed space $(c_{00},|||\cdot|||)$ by $(E,|||\cdot|||)$. From this, $E$ is the completion of $\langle\{e_i:i\in\N\}\rangle$. We propose $(e_i)_{i\in\N}$ as the required $(\mathcal{B}_i\upharpoonright_{M})_{i\in\N}-$block asymptotic model of $(x_i)_{i\in M}$. First of all, let us show that $(e_i)_{i\in\N}$ is a Schauder basis for $(E,|||\cdot|||)$. Certainly, since $(x_i)_{i\in M}$ is a basic sequence, there is $C\geq 1$ such that for any $m,n\in\N$ with $m\leq n$ and for every $S\in{Bl}(\mathcal{B}_1\upharpoonright_{M},\ldots,\mathcal{B}_{n}\upharpoonright_{M})$ one obtains that
	\begin{align*}
	\Psi_m(\{s_1,\ldots,s_m\})(\sum_{i=1}^m a_ie_i)=\|\sum_{i=1}^m a_i\mathcal{X}(s_i)\|\leq C\|\sum_{i=1}^n a_i\mathcal{X}(s_i)\|=C\Psi_n(S)(\sum_{i=1}^n a_ie_i)
	\end{align*}
	for all $(a_i)_{i=1}^n\in[-1,1]^n$. From this it follows that 
	\begin{align*}
	|||\sum_{i=1}^m a_ie_i|||=\rho_m(\sum_{i=1}^m a_ie_i)\leq C\rho_n(\sum_{i=1}^n a_ie_i)=C |||\sum_{i=1}^n a_ie_i|||
	\end{align*}
	holds for each $m,n\in\N$ with $m\leq n$ and each $(a_i)_{i=1}^n\in[-1,1]^n$. Hence, $(e_i)_{i\in\N}$ is a Schauder basis for $(E,|||\cdot|||)$. Finally, we prove that $(e_i)_{i\in\N}$ satisfies the remained condition to be a $(\mathcal{B}_i\upharpoonright_{M})_{i\in\N}-$block asymptotic model of $(x_i)_{i\in M}$. To have this done, we fix $k\in\N$ and $S\in{Bl}(\mathcal{B}_1\upharpoonright_{M/m_{k-1}},$ $\ldots,\mathcal{B}_k\upharpoonright_{M/m_{k-1}})$, where $\min (s_1)=m_{\ell}$ for some $\ell\in\N$. According to the convergence (\ref{BaBkconvergencia}), there is $n\in \N$ such that for every $T=\{t_1,\ldots,t_k\}\in{Bl}(\mathcal{B}_1\upharpoonright_{M/n},$ $\ldots,\mathcal{B}_k\upharpoonright_{M/n})$ we get that
	\begin{equation}\label{BaBSdes2}
	\big|\|\sum_{i=1}^ka_i\mathcal{X}(t_i)\| -|||\sum_{i = 1}^k a_i e_i ||| \big|<\frac{\varepsilon_{\ell}}{2}
	\end{equation}
	for each $(a_i)_{i=1}^k\in[-1,1]^k$. If $T\in{Bl}(\mathcal{B}_1\upharpoonright_{M/\max\{n,m_{\ell}\}},\ldots,\mathcal{B}_k\upharpoonright_{M/\max\{n,m_{\ell}\}})$, then (\ref{BaBSdes1}) and (\ref{BaBSdes2}) imply that
	\begin{align*}
	\big| \| \sum_{i=1}^{k}a_i \mathcal{X}(s_i) \| - ||| \sum_{i = 1}^k a_i e_{i} |||\big|
	&\leq   \big| \| \sum_{i=1}^{k}a_i \mathcal{X}(s_i) \| - \| \sum_{i=1}^{k} a_i\mathcal{X}(t_i) \|\big| \\ 
	&\quad\quad +
	\big| \| \sum_{i=1}^{k} a_i \mathcal{X}(t_i) \| - ||| \sum_{i = 1}^k a_i e_i ||| \big|<\varepsilon_{\ell}  
	\end{align*}
	for each $(a_i)_{i=1}^k\in[-1,1]^k$, as required.
\end{proof}

The next corollary generalizes Corollary 3.4 of the article \cite{cg}.

\begin{cor}
	Ramsey Theorem is equivalent to Theorem \ref{BrunelSuchestonBBarreras}.
\end{cor}
\begin{proof}
	Because of the Theorem \ref{TeoPrincipalBarreras}, it is enough to prove that clause (7) of Theorem \ref{TeoPrincipalBarreras} and Theorem \ref{BrunelSuchestonBBarreras} are equivalent.
	The first implication follows from the proof of Theorem \ref{BrunelSuchestonBBarreras}. For the converse, let $(x_i)_{i\in\N}$ be a normalized basic sequence in a Banach space $(X,\|\cdot\|)$ and let $(\mathcal{B}_i)_{i\in\N}$ be a sequence of barriers on $\N$. According to the assumption, there are $ \varepsilon_i\searrow0$ and $M\in[\N]^{\omega}$ such that $(x_i)_{i\in M}$ induces a norm $|||\cdot|||$ on the linear space $c_{00}$ and for every $k\in\N$ and $S=\{s_1,\ldots,s_k\}\in Bl(\mathcal{B}_1\upharpoonleft_{M/m_{k-1}},\ldots,(\mathcal{B}_k\upharpoonleft_{M/m_{k-1}})$ we get that
	\begin{align*}
	\big| \|\sum_{i=1}^ka_i\mathcal{E}(s_i)\|-|||\sum_{i=1}^ka_ie_i|||\big|<\varepsilon_{\ell},
	\end{align*}
	where $\min(s_1)=m_{\ell}$, for all $(a_i)_{i=1}^k\in[-1,1]^k$. By the triangle inequality, we obtain that
	\begin{align*}
		\big| \|\sum_{i=1}^ka_i\mathcal{E}(s_i)\|-\|\sum_{i=1}^ka_i\mathcal{E}(t_i)\|\big|&\leq \big|\|\sum_{i=1}^ka_i\mathcal{E}(s_i)\|-|||\sum_{i=1}^ka_ie_i|||\big|\\
		& \quad\quad +\big| |||\sum_{i=1}^ka_ie_i|||-\|\sum_{i=1}^ka_i\mathcal{E}(t_i)\|\big|\leq2\varepsilon_{\ell}
	\end{align*}
	whenever $\min(s_1\cup t_1)=m_{\ell}$, for each $k\in\N$, $S,$ $T\in Bl(\mathcal{B}_1\upharpoonleft_{M/m_{k-1}},\ldots,(\mathcal{B}_k\upharpoonleft_{M/m_{k-1}})$ and $(a_i)_{i=1}^k\in[-1,1]^k$. Hence, the sequence $(x_i)_{i\in M}$ is $(\mathcal{B}_i\upharpoonleft_{M})_{i\in\N}-$block asymptotic oscillation stable.
\end{proof}

A direct consequence of the previous theorem is the fact that a normalized basic sequence is able to generate several block asymptotic models, each one related to a different sequence of barriers on $\N$. This lead us to pose the next natural question
that arose in the process of this work.

\begin{question}
	What kinds of properties can distinguish one from the other the block asymptotic models of the same normalized basic sequence?
\end{question} 

This question motivates a new research line that consists of the study of the various asymptotic block models of the same normalized basic sequence. An example of this is presented in the next section where two asymptotic models will be described, one of which is spreading and the other is not.

\medskip

For the next question we consider the Tsirelson space \cite{Tsirelson} denoted by $T$ and the fact that every spreading model of $T$ is isomorphic to $\ell_1$ (for a proof we refer a reader to \cite[Th. 3.5.(3)]{Od2}).

\begin{question}
	Is every $\mathcal{B}-$spreading model of $T$ isomorphic to $\ell_1$? where $\mathcal{B}$ is an arbitrary barrier.
\end{question}

The iteration applied to spreading models has been studied in the article \cite{Ar}. This idea can be applied to the block asymptotic models. In particular, for a barrier $\mathcal{B}$ on $\N$ one could define the $(2,\mathcal{B})-${\it iterated spreading model} of a normalized basic sequence $(x_i)_{i\in\N}$ as the $\mathcal{B}$-spreading model of a subsequence of the $\mathcal{B}$-spreading model of a subsequence of $(x_i)_{i\in\N}$. For every $k>2$, the $(k,\mathcal{B})-$iterated spreading model of a normalized basic sequence is defined analogously. It is shown in \cite[Cor. 69]{Ar} that there exists a Banach space $X$ with a normalized basic sequence $(x_i)_{i\in\N}$ such that the $k-$iterated spreading model of $(x_i)_{i\in\N}$ is not isomorphic to neither $\ell_p$ ($1\leq p <\infty$) nor $c_0$, for every $k\in\N$. In this direction, we formulate the following questions.

\begin{question}
	Given a barrier $\mathcal{B}$ on $\N$, is there a normalized basic sequence so that, for every $k\in\N$, its $(k,\mathcal{B})-$iterated spreading model is not isomorphic to $\ell_p$ $(1\leq p<\in\infty)$ or $c_0$?
\end{question}

\begin{question}
	Given $k\geq 2$ and a barrier $\mathcal{B}$ on $\N$, does there exist a normalized basic sequence such that its $(k,\mathcal{B})-$iterated spreading model is not isomorphic to $\ell_p$ $(1\leq p<\infty)$ or $c_0$?
\end{question} 

The previous questions could be extended by the notion of iteration that alternates the elements of a sequence of barriers $(\mathcal{B}_i)_{i\in\N}$, i.e., in the $i-$th case one gets the $\mathcal {B}_i$-spreading model of the sequence obtained from the preceding step.

\section{Example}

The aim of this section is to obtain two block asymptotic models of the same normalized basic sequence that are different but equivalent and only one of them has the spreading property.

\medskip 

Our example is based on the normed space $(c_{00},\|\cdot\|)$, whose norm is defined by
\begin{align*}
\|x\|=\sup&\Big(\big\{|a_i|:i\in\N\big\}\cup\big\{\frac{3}{4}\sum_{i\in s}|a_i|:s\in[\N]^2\big\}\cup\big\{\frac{9}{16}\sum_{i\in s}|a_i|: s\in[\N]^8\big\}\Big)
\end{align*}
for all $x=\sum_{i=1}^{\infty}a_ie_i\in c_{00}$. The completion of this normed space will be denoted by $(X,\|\cdot\|)$.

\begin{obs}\label{ObservacionEjemploModeloDisperso}
	For $k\in\N$ and $(a_i)_{i=1}^k\in[-1,1]^k$, we have the following relationships which can be checked easily:
	\begin{enumerate}[label=\textnormal{(\roman*)}]
		\item $\|e_i\|=1$ for all $i\in\N$,\label{normalized}
		\item (Unconditional) 
		$\|\sum_{i=1}^k a_ie_{i}\|=\|\sum_{i=1}^k\varepsilon_ia_{i}e_{i}\|$
		for all $(\varepsilon_i)_{i=1}^k\in\{-1,1\}^k$,\label{positive}
		\item $\|\sum_{i=1}^k a_ie_{i}\|=\|\sum_{i=1}^ka_{\sigma(i)}e_{i}\|$ for each permutation $\sigma:\{1,\ldots,k\}\to\{1,\ldots,k\}$,\label{ordered} 
		\item $\|\sum_{i=1}^k a_ie_{i}\|=\|\sum_{i=1}^k a_{i}e_{t(i)}\|$ for every $t\in[\N]^k$, and\label{spreading}
		\item $\|\sum_{i=1}^{\ell} a_ie_{i}\|\leq\|\sum_{i=1}^k a_{i}e_{i}\|$ for all $\ell\leq k$.\label{Schauder}
	\end{enumerate}
\end{obs}
From \ref{Schauder} it follows that $(e_i)_{i\in\N}$ is a Schauder basis with constant $1$ for $(X,\|\cdot\|)$. On the other hand, \ref{spreading} says that $(e_i)_{i\in\N}$ is a spreading sequence in $(X,\|\cdot\|)$. Henceforth, because of \ref{positive}, we may assume without loss of generality that $(a_i)_{i=1}^k\in[ 0,1]^k$ for all $k\in\N$.

\medskip 

Now, for all $k\in\N$, $t\in[\N]^k$ and $(a_i)_{i=1}^k\in[0,1]^k$ with $a_1\geq a_2\geq\cdots\geq a_k$ we can calculate the norm of  $\sum_{i=1}^k a_ie_{t(i)}$ as follows:
\begin{align}\label{Norma28}
\|\sum_{i=1}^k a_ie_{t(i)}\|&=
\left\{
\begin{array}{l l}
a_1&\text{if}~ a_1\geq\frac{3}{4}\sum_{i=1}^2a_i\\
&~\text{and}~ a_1\geq\frac{9}{16}\sum_{i=1}^{\min\{k,8\}}a_i,\\
\frac{3}{4}\sum_{i=1}^2a_i&\text{if}~ a_1\leq\frac{3}{4}\sum_{i=1}^2a_i\\
&~\text{and}~ \frac{3}{4}\sum_{i=1}^2a_i\geq\frac{9}{16}\sum_{i=1}^{\min\{k,8\}}a_i,\\
\frac{9}{16}\sum_{i=1}^{\min\{k,8\}}a_i&\text{if}~ a_1\leq 	\frac{9}{16}\sum_{i=1}^{\min\{k,8\}}a_i\\
&~\text{and}~ \frac{3}{4}\sum_{i=1}^2a_i\leq\frac{9}{16}\sum_{i=1}^{\min\{k,8\}}a_i,
\end{array}
\right.,
\end{align}
in particular,
\begin{align*}
\|\sum_{i=1}^ke_{t(i)}\|=\left\{
\begin{array}{ll}
\frac{3}{2}&\text{if}~k=2\\
\frac{9k}{16}&\text{if}~2 <k<8\\
\frac{9}{2}&\text{if}~8\geq k
\end{array}
\right..
\end{align*}

We will apply Theorem \ref{BrunelSuchestonBBarreras} to obtain two block asymptotic models associated to $([\N]^8,\ldots,$ $[\N]^8,\ldots)$ and $([\N]^2,[\N]^2,[\N]^8,\ldots,[\N]^8,\ldots)$, respectively, but first we calculate suitable norms which are important for final results.

\medskip

We recall the notation $\mathcal{E}(s)=\frac{\sum_{i\in s}e_i}{\|\sum_{i\in s}e_i\|}$ for all $s\in FIN^*$. 

\medskip

First of all, we will estimate a lower bound and an upper bound of $\|\sum_{i=1}^ka_i\mathcal{E}(s_i)\|$ for all $k\in\N$, $S=\{s_1,\ldots,s_k\}\in Bl([M]^2,[M]^2,[M]^8,\ldots,[M]^8)$ and $(a_i)_{i=1}^k\in[0,1]^k$. We will proceed increasingly, starting with the case $k=2$ and in all instances we will use the formula  (\ref{Norma28}):

\medskip


For every $S=\{s_1,s_2\}\in Bl^2([\N]^2)$ and $(a_i)_{i=1}^2\in[0,1]^2$ with $a_1\geq a_2$ we have
\begin{flalign*}
\|a_1\mathcal{E}(s_1)+a_2\mathcal{E}(s_2)\|
&=\|a_1\left(\frac{\sum_{j=1}^2e_{s_1(j)}}{\|\sum_{j=1}^2e_{s_1(j)}\|}\right)+ a_2\left(\frac{\sum_{j=1}^2e_{s_2(j)}}{\|\sum_{j=1}^2e_{s_2(j)}\|}\right)\|\notag\\
&=\frac{2}{3}\|\sum_{j=1}^2a_1e_{s_1(j)}+ \sum_{j=1}^2a_2e_{s_2(j)}\|\notag\\
&=\frac{2}{3}\left\{
\begin{array}{ll}
\frac{3}{2}a_1&\text{if}~\frac{3}{2}a_1\geq \frac{9}{8}(a_1+a_2)\notag\\
\frac{9}{8}(a_1+a_2)&\text{if}~\frac{3}{2}a_1< \frac{9}{8}(a_1+a_2)
\end{array}
\right.\notag\\
&=\left\{
\begin{array}{ll}
a_1&\text{if}~\frac{3}{2}a_1\geq \frac{9}{8}(a_1+a_2)\\
\frac{3}{4}(a_1+a_2)&\text{if}~\frac{3}{2}a_1< \frac{9}{8}(a_1+a_2)
\end{array}
\right..
\end{flalign*}
By property \ref{ordered} we get the general formula
\begin{align}\label{normaB228-2} 
\|a_1\mathcal{E}(s_1)+a_2\mathcal{E}(s_2)\|=\left\{
\begin{array}{ll}
\max\{a_1,a_2\}&\text{if}~\frac{3}{2}\max\{a_1,a_2\}\geq \frac{9}{8}(a_1+a_2)\\
\frac{3}{4}(a_1+a_2)&\text{if}~\frac{3}{2}\max\{a_1,a_2\}< \frac{9}{8}(a_1+a_2)
\end{array}
\right.
\end{align}
for all $S=\{s_1,s_2\}\in Bl^2([\N]^2)$ and $(a_i)_{i=1}^2\in[0,1]^2$.
Also, by the above equation, we deduce that for all $S\in Bl^2([\N]^2)$ and $(a_i)_{i=1}^2\in[0,1]^2$
\begin{align}\label{desk=2}
\max\{a_1,a_2\}\leq\|a_1\mathcal{E}(s_1)+a_2\mathcal{E}(s_2)\|\leq\frac{3}{2}\max\{a_1,a_2\}.
\end{align}

\medskip

Now, for every $S=\{s_1,s_2,s_3\}\in Bl([\N]^2,[\N]^2,[\N]^8)$ and for every $(a_i)_{i=1}^3\in[0,1]^3$ we obtain
\begin{flalign}
\|\sum_{i=1}^3a_i\mathcal{E}(s_i)\|
&=\|a_1\left(\frac{\sum_{j=1}^2e_{s_1(j)}}{\|\sum_{j=1}^2e_{s_1(j)}\|}\right)+ a_2\left(\frac{\sum_{j=1}^2e_{s_2(j)}}{\|\sum_{j=1}^2e_{s_2(j)}\|}\right)+a_3\left(\frac{\sum_{j=1}^8e_{s_3(j)}}{\|\sum_{j=1}^8e_{s_3(j)}\|}\right)\|\notag\\
&=\frac{2}{3}\|\sum_{j=1}^2a_1e_{s_1(j)}+ \sum_{j=1}^2a_2e_{s_2(j)}+\sum_{j  =1}^8\frac{1}{3}a_3e_{s_3(j)}\|
\notag\\
&=\frac{2}{3}\left\{
\begin{array}{ll}
\frac{3}{2}a_1&\text{if}~a_1\geq a_2\geq\frac{1}{3}a_3~\text{and}~\frac{3}{2}a_1\geq \frac{9}{8}(a_1+a_2+\frac{2}{3}a_3)\\
\frac{9}{8}(a_1+a_2+\frac{2}{3}a_3)&\text{if}~a_1\geq a_2\geq\frac{1}{3}a_3~\text{and}~\frac{3}{2}a_1< \frac{9}{8}(a_1+a_2+\frac{2}{3}a_3)\\
\frac{3}{2}a_1&\text{if}~a_1\geq\frac{1}{3}a_3\geq a_2~\text{and}~\frac{3}{2}a_1\geq \frac{9}{8}(a_1+a_3)\\
\frac{9}{8}(a_1+a_3)&\text{if}~a_1\geq\frac{1}{3}a_3\geq a_2~\text{and}~\frac{3}{2}a_1<\frac{9}{8}(a_1+a_3)\\
\frac{3}{2}a_2&\text{if}~a_2\geq a_1\geq\frac{1}{3}a_3~\text{and}~\frac{3}{2}a_2\geq \frac{9}{8}(a_1+a_2+\frac{2}{3}a_3)\\
\frac{9}{8}(a_1+a_2+\frac{2}{3}a_3)&\text{if}~a_2\geq a_1\geq\frac{1}{3}a_3~\text{and}~\frac{3}{2}a_2< \frac{9}{8}(a_1+a_2+\frac{2}{3}a_3)\\
\frac{3}{2}a_2&\text{if}~a_2\geq\frac{1}{3}a_3\geq a_1~\text{and}~\frac{3}{2}a_2\geq \frac{9}{8}(a_2+a_3)\\
\frac{9}{8}(a_2+a_3)&\text{if}~a_2\geq\frac{1}{3}a_3\geq a_1~\text{and}~\frac{3}{2}a_2<\frac{9}{8}(a_2+a_3)\\
\frac{3}{2}a_3&\text{if}~ \frac{1}{3}a_3\geq a_1,a_2
\end{array}
\right.\notag\\
&=~~\left\{
\begin{array}{ll}
a_1&\text{if}~a_1\geq a_2\geq\frac{1}{3}a_3~\text{and}~\frac{3}{2}a_1\geq \frac{9}{8}(a_1+a_2+\frac{2}{3}a_3)\\
\frac{3}{4}(a_1+a_2+\frac{2}{3}a_3)&\text{if}~a_1\geq a_2\geq\frac{1}{3}a_3~\text{and}~\frac{3}{2}a_1< \frac{9}{8}(a_1+a_2+\frac{2}{3}a_3)\\
a_1&\text{if}~a_1\geq\frac{1}{3}a_3\geq a_2~\text{and}~\frac{3}{2}a_1\geq \frac{9}{8}(a_1+a_3)\\
\frac{3}{4}(a_1+a_3)&\text{if}~a_1\geq\frac{1}{3}a_3\geq a_2~\text{and}~\frac{3}{2}a_1<\frac{9}{8}(a_1+a_3)\\
a_2&\text{if}~a_2\geq a_1\geq\frac{1}{3}a_3~\text{and}~\frac{3}{2}a_2\geq \frac{9}{8}(a_1+a_2+\frac{2}{3}a_3)\\
\frac{3}{4}(a_1+a_2+\frac{2}{3}a_3)&\text{if}~a_2\geq a_1\geq\frac{1}{3}a_3~\text{and}~\frac{3}{2}a_2< \frac{9}{8}(a_1+a_2+\frac{2}{3}a_3)\\
a_2&\text{if}~a_2\geq\frac{1}{3}a_3\geq a_1~\text{and}~\frac{3}{2}a_2\geq \frac{9}{8}(a_2+a_3)\\
\frac{3}{4}(a_2+a_3)&\text{if}~a_2\geq\frac{1}{3}a_3\geq a_1~\text{and}~\frac{3}{2}a_2<\frac{9}{8}(a_2+a_3)\\
a_3&\text{if}~ \frac{1}{3}a_3\geq a_1,a_2
\end{array}
\right..\label{normaB228-3}
\end{flalign}

\medskip

Next, we proceed to find a lower bound and a upper bound for the previous norm. To have this done, fix $S\in Bl([\N]^2,[\N]^2,[\N]^8)$ and $(a_i)_{i=1}^3\in[0,1]^3$, and analyze all possible orderings of the set $\{a_1,a_2,a_3\}$. Let us start with the two main cases and we split them in suitable subcases:

\medskip

Case I. We first suppose that $a_3=\max\{a_i:i\leq 3\}$. It follows from (\ref{normaB228-3}) that the only possible subcases are the following:

I$.1.$ For $a_k\geq a_{\ell}\geq\frac{1}{3}a_3$ and $\frac{3}{2}a_k< \frac{9}{8}(a_k+a_{\ell}+\frac{2}{3}a_3)$ where $k,\ell\in\{1,2\}$ and $k\neq \ell$, since
$a_3\geq a_1,a_2$ and $a_1,a_2\geq\frac{1}{3}a_3$ one gets
\begin{align*}
a_3=\frac{3}{4}(\frac{1}{3}a_3+\frac{1}{3}a_3+\frac{2}{3}a_3)\leq\|\sum_{i=1}^3a_i\mathcal{E}(s_i)\|=\frac{3}{4}(a_1+a_2+\frac{2}{3}a_3)\leq\frac{3}{4}(a_3+a_3+\frac{2}{3}a_3)=2a_3.
\end{align*}

I$.2.$ For $a_k\geq\frac{1}{3}a_3\geq a_{\ell}$ and $\frac{3}{2}a_k<\frac{9}{8}(a_k+a_3)$  where $k,\ell\in\{1,2\}$ and $k\neq \ell$, we get
\begin{align*}
a_3=\frac{3}{4}(\frac{1}{3}a_3+a_3)\leq\|\sum_{i=1}^3a_i\mathcal{E}(s_i)\|=\frac{3}{4}(a_k+a_3)\leq\frac{3}{4}(a_3+a_3)=\frac{3}{2}a_3.
\end{align*}

I$.3.$ For $\frac{1}{3}a_3\geq a_1,a_2$ we know that $\|\sum_{i=1}^3a_i\mathcal{E}(s_i)\|=a_3$.

\noindent Hence, we have
\begin{align*}
\max\{a_i:i\leq 3\}=a_3\leq\|\sum_{i=1}^3a_i\mathcal{E}(s_i)\|\leq 2a_3=2\max\{a_i:i\leq 3\}.
\end{align*}

Case II. We assume that $a_k=\max\{a_i:i\leq 3\}$ for some $k\in\{1,2\}$. We compute the norm of $\sum_{i=1}^3a_i\mathcal{E}(s_i)$ through the following subcases:

II$.1.$ If $a_k\geq a_{\ell}\geq\frac{1}{3}a_3$ and $\frac{3}{2}a_k\geq \frac{9}{8}(a_k+a_{\ell}+\frac{2}{3}a_3)$ with $\ell\in\{1,2\}$ and $k\neq \ell$, then
\begin{align*}
\|\sum_{i=1}^3a_i\mathcal{E}(s_i)\|=a_k.
\end{align*}

II$.2.$ For $a_k\geq a_{\ell}\geq\frac{1}{3}a_3$ and $\frac{3}{2}a_k< \frac{9}{8}(a_k+a_{\ell}+\frac{2}{3}a_3)$ with $\ell\in\{1,2\}$ and $k\neq \ell$, directly from the initial assumption one obtains
\begin{align*}
\|\sum_{i=1}^3a_i\mathcal{E}(s_i)\|=\frac{3}{4}(a_1+a_2+\frac{2}{3}a_3)\leq\frac{3}{4}(a_k+a_k+\frac{2}{3}a_k)=2a_k.
\end{align*}
From the inequality $\frac{3}{2}a_k<\frac{9}{8}(a_k+a_{\ell}+\frac{2}{3}a_3)$, it is possible to deduce $\frac{1}{3}a_k<a_{\ell}+\frac{2}{3}a_3$ and so
\begin{align*}
a_k=\frac{3}{4}(a_k+\frac{1}{3}a_k)\leq\frac{3}{4}(a_k+a_{\ell}+\frac{2}{3}a_3)=\|\sum_{i=1}^3a_i\mathcal{E}(s_i)\|.
\end{align*}

II$.3.$ If $a_k\geq\frac{1}{3}a_3\geq a_{\ell}$ and $\frac{3}{2}a_k\geq \frac{9}{8}(a_k+a_3)$ with $\ell\in\{1,2\}$ and $k\neq \ell$, then
\begin{align*}
\|\sum_{i=1}^3a_i\mathcal{E}(s_i)\|=a_k.
\end{align*}

II$.4.$ For $a_k\geq\frac{1}{3}a_3\geq a_{\ell}$ and $\frac{3}{2}a_k< \frac{9}{8}(a_k+a_3)$ with $\ell\in\{1,2\}$ and $k\neq \ell$, we find the inequality 
\begin{align*}
\|\sum_{i=1}^3a_i\mathcal{E}(s_i)\|=\frac{3}{4}(a_k+a_3)\leq\frac{3}{4}(a_k+a_k)=\frac{3}{2}a_k,
\end{align*}
since $a_k\geq a_3$. The condition $\frac{3}{2}a_k< \frac{9}{8}(a_k+a_3)$ implies that $\frac{1}{3}a_k<a_3$ and so
\begin{align*}
a_k=\frac{3}{4}(a_k+\frac{1}{3}a_k)\leq\frac{3}{4}(a_k+a_3)=\|\sum_{i=1}^3a_i\mathcal{E}(s_i)\|.
\end{align*}	
\noindent All the subcases of the second case lead us to the relationship
\begin{align*}
\max\{a_i:i\leq 3\}=a_k\leq\|\sum_{i=1}^3a_i\mathcal{E}(s_i)\|\leq 2a_k=2\max\{a_i:i\leq 3\}.
\end{align*}

After analyzing all possible instances, we obtain
\begin{align}\label{desk=3}
\max\{a_i:i\leq 3\}\leq\|\sum_{i=1}^3a_i\mathcal{E}(s_i)\|\leq2\max\{a_i:i\leq 3\}
\end{align}
for all $S\in Bl([\N]^2,[\N]^2,[\N]^8)$ and $(a_i)_{i=1}^3\in[0,1]^3$.

\medskip

Now, we will get an inequality similar to the previous one for each element of $Bl([\N]^2,[\N]^2,$ $[\N]^8, [\N]^8)$. Fix $S=\{s_1,s_2,s_3,s_4\}\in Bl([\N]^2,[\N]^2,[\N]^8,[\N]^8)$ and observe that
\begin{flalign*}
\|\sum_{i=1}^4a_i\mathcal{E}(s_i)\|&=\|\sum_{i=1}^2a_i\left(\frac{\sum_{j=1}^2e_{s_i(j)}}{\|\sum_{j=1}^2e_{s_i(j)}\|}\right)+\sum_{i=3}^4a_i\left(\frac{\sum_{j=1}^8e_{s_i(j)}}{\|\sum_{j=1}^8e_{s_i(j)}\|}\right)\|\\
&=\frac{2}{3}\|\sum_{j=1}^2a_1e_{s_1(j)}+\sum_{j=1}^2a_2e_{s_2(j)}+\sum_{j  =1}^8\frac{1}{3}a_3e_{s_3(j)}+\sum_{j=1}^8\frac{1}{3}a_4e_{s_4(j)})\|
\end{flalign*}
for every $(a_i)_{i=1}^4\in[0,1]^4$. To continue we will analyze the norm
\begin{align*}
\|\sum_{j=1}^2a_1e_{s_1(j)}+\sum_{j=1}^2a_2e_{s_2(j)}+\sum_{j  =1}^8\frac{1}{3}a_3e_{s_3(j)}+\sum_{j=1}^8\frac{1}{3}a_4e_{s_4(j)}\|
\end{align*}
for all $(a_i)_{i=1}^4\in[0,1]^4$. In order to do this, we consider two cases for each $(a_i)_{i=1}^4\in[0,1]^4$:

Case I. Assume $a_4\geq a_3$ and fix $s\in[\N]^{\ell}$ for some $\ell\in\{1,2,8\}$. Notice that if $s\cap s_3=\emptyset$, then
\begin{align*}
\frac{\ell+1}{2\ell}\Big(\sum_{j=1}^2|a_1|\chi_{s}(s_1(j))&+\sum_{j=1}^2|a_2|\chi_{s}(s_2(j))+\sum_{j=1}^8\frac{1}{3}|a_3|\chi_{s}(s_3(j))+\sum_{j=1}^8\frac{1}{3}|a_4|\chi_{s}(s_4(j))\Big)\\
&=\frac{\ell+1}{2\ell}\Big(\sum_{j=1}^2|a_1|\chi_{s}(s_1(j))+\sum_{j=1}^2|a_2|\chi_{s}(s_3(j))+\sum_{j=1}^8\frac{1}{3}|a_4|\chi_{s}(s_4(j))\Big)\\
&\leq \|\sum_{j=1}^2a_1e_{s_1(j)}+\sum_{j=1}^2a_2e_{s_2(j)}+\sum_{j  =1}^8\frac{1}{3}a_4e_{s_4(j)}\|.
\end{align*}
When $s\cap s_3\neq\emptyset$, we consider $t=(s\setminus s_3)\cup u$ for some $u\subseteq s_4\backslash s$ with $|s\cap s_3|$ so that
\begin{align*}
&\frac{\ell+1}{2\ell}\Big(\sum_{j=1}^2|a_1|\chi_{s}(s_1(j))+\sum_{j=1}^2|a_2|\chi_{s}(s_2(j))+\sum_{j=1}^8\frac{1}{3}|a_3|\chi_{s}(s_3(j))+\sum_{j=1}^8\frac{1}{3}|a_4|\chi_{s}(s_4(j))\Big)
\\
&\leq\frac{\ell+1}{2\ell}\Big(\sum_{j=1}^2|a_1|\chi_{t}(s_1(j))+\sum_{j=1}^2|a_2|\chi_{t}(s_2(j))+\sum_{j=1}^8\frac{1}{3}|a_3|\chi_{t}(s_3(j))+\sum_{j=1}^8\frac{1}{3}|a_4|\chi_{t}(s_4(j))\Big)
\\
&=\frac{\ell+1}{2\ell}\Big(\sum_{j=1}^2|a_1|\chi_{t}(s_1(j))+\sum_{j=1}^2|a_2|\chi_{t}(s_2(j))+\sum_{j=1}^8\frac{1}{3}|a_4|\chi_{t}(s_4(j))\Big)
\\
&\leq \|\sum_{j=1}^2a_1e_{s_1(j)}+\sum_{j=1}^2a_2e_{s_2(j)}+\sum_{j  =1}^8\frac{1}{3}a_4e_{s_4(j)}\|.
\end{align*}
From this it follows that
\begin{align*}
\|\sum_{j=1}^2a_1e_{s_1(j)}+&\sum_{j=1}^2a_2e_{s_2(j)}+\sum_{j=1}^8\frac{1}{3}a_3e_{s_3(j)}+\sum_{j=1}^8\frac{1}{3}a_4e_{s_4(j)}\|\\
&\leq\|\sum_{j=1}^2a_1e_{s_1(j)}+\sum_{j=1}^2a_2e_{s_2(j)}+\sum_{j=1}^8\frac{1}{3}a_4e_{s_4(j)}\|.
\end{align*}
Since $(e_i)_{i\in\N}$ is a Schauder basis for $(X,\|\cdot\|)$ with basis constant 1 and a spreading sequence, if we take $s\in[\N]^8$ with $s>s_4$ then
\begin{align*}
\|\sum_{j=1}^2a_1e_{s_1(j)}+&\sum_{j=1}^2a_2e_{s_2(j)}+\sum_{j=1}^8\frac{1}{3}a_3e_{s_3(j)}+\sum_{j=1}^8\frac{1}{3}a_4e_{s_4(j)}\|\\&=
\|\sum_{j=1}^2a_1e_{s_1(j)}+\sum_{j=1}^2a_2e_{s_2(j)}+\sum_{j=1}^8\frac{1}{3}a_4e_{s_4(j)}+\sum_{j=1}^8\frac{1}{3}a_3e_{s(j)}\|
\\
&\geq\|\sum_{j=1}^2a_1e_{s_1(j)}+\sum_{j=1}^2a_2e_{s_2(j)}+\sum_{j=1}^8\frac{1}{3}a_4e_{s_4(j)}\|.
\end{align*}
Hence
\begin{align*}
\|\sum_{j=1}^2a_1e_{s_1(j)}&+ \sum_{j=1}^2a_2e_{s_2(j)}+\sum_{j  =1}^8\frac{1}{3}a_3e_{s_3(j)})+\sum_{j  =1}^8\frac{1}{3}a_4e_{s_4(j)}\|\\
&=\|\sum_{j=1}^2a_1e_{s_1(j)}+ \sum_{j=1}^2a_2e_{s_2(j)}+\sum_{j  =1}^8\frac{1}{3}a_3e_{s_3(j)}\|.
\end{align*}
\\
Case II. For the case $a_3\geq a_4$, proceeding as in the previous case, we get that
\begin{align*}
\|\sum_{j=1}^2a_1e_{s_1(j)}&+ \sum_{j=1}^2a_2e_{s_2(j)}+\sum_{j  =1}^8\frac{1}{3}a_3e_{s_3(j)}+\sum_{j  =1}^8\frac{1}{3}a_4e_{s_4(j)}\|\\
&=\|\sum_{j=1}^2a_1e_{s_1(j)}+ \sum_{j=1}^2a_2e_{s_2(j)}+\sum_{j  =1}^8\frac{1}{3}a_3e_{s_3(j)}\|.
\end{align*}

Therefore, for each $S\in Bl([\N]^2,[\N]^2,[\N]^8,[\N]^8)$ and $(a_i)_{i=1}^4\in[0,1]^4$ we have
\begin{flalign}\label{normaB2288}
\|\sum_{i=1}^4a_i\mathcal{E}(s_i)\|=\left\{
\begin{array}{ll}
\|a_1\mathcal{E}(s_1)+a_2\mathcal{E}(s_2)+a_3\mathcal{E}(s_3)\|&\text{if}~a_3\geq a_4\\
\|a_1\mathcal{E}(s_1)+a_2\mathcal{E}(s_2)+a_4\mathcal{E}(s_4)\|&\text{if}~a_4\geq a_3  
\end{array}
\right..
\end{flalign}

From the previous identity and (\ref{desk=3}) it follows that for every $S\in Bl([\N]^2,[\N]^2,$ $[\N]^8,[\N]^8)$ and $(a_i)_{i=1}^4\in[0,1]^4$
\begin{align}\label{desk=4}
\max\{a_i:i\leq 4\}\leq  \|\sum_{i=1}^4a_i\mathcal{E}(s_i)\|\leq2 \max\{a_i:i\leq 4\}.
\end{align}

\medskip

Modifying the procedure of the case $k=4$  is possible to obtain some characteristics about the norms of suitable vectors:

\noindent For each $k\in\N/3$ and $S=\{s_1,\ldots,s_k\}\in Bl([\N]^2,[\N]^2,[\N]^8,\ldots,[\N]^8)$ we have
\begin{flalign}\label{normaB2288...}
\|\sum_{i=1}^ka_i\mathcal{E}(s_i)\|
=\left\{
\begin{array}{ll}
\|a_1\mathcal{E}(s_1)+a_2\mathcal{E}(s_2)+a_3\mathcal{E}(s_3)\|&\text{if}~a_3=\max\{a_j:~ 3\leq j\leq k\}\\
\|a_1\mathcal{E}(s_1)+a_2\mathcal{E}(s_2)+a_4\mathcal{E}(s_4)\|&\text{if}~a_4=\max\{a_j:~ 3\leq j\leq k\}
\\
&\vdots\\
\|a_1\mathcal{E}(s_1)+a_2\mathcal{E}(s_2)+a_k\mathcal{E}(s_k)\|&\text{if}~ a_k=\max\{a_j:~ 3\leq j\leq k\}
\end{array}
\right.,
\end{flalign}
and so, by (\ref{desk=3}), we obtain
\begin{align}\label{desk}
\max\{a_i:i\leq k\}\leq  \|\sum_{i=1}^ka_i\mathcal{E}(s_i)\|\leq2 \max\{a_i:i\leq k\},
\end{align}
for all $(a_i)_{i=1}^k\in[0,1]^k$.

\addvspace{\medskipamount}


Taking in account (\ref{Norma28}), for every $k\in\N$, $T=\{t_1,\ldots,t_k\}\in Bl^k([\N]^8)$ and $(a_i)_{i=1}^k\in[0,1]^k$ with $a_1\geq a_2\geq\cdots\geq a_k$ we get
\begin{flalign*}
\|\sum_{i=1}^ka_i\mathcal{E}(t_i)\|
&=\|\sum_{i=1}^ka_i\left(\frac{\sum_{j=1}^8e_{t_i(j)}}{\|\sum_{j=1}^8e_{t_i(j)}\|}\right)\|
=\frac{2}{9}\|\sum_{i=1}^k\sum_{j=1}^8a_ie_{t_i(j)}\|\notag\\
&=\frac{2}{9}\left(\frac{9}{2}a_1\right)=a_1.
\end{flalign*}
Hence, by property \ref{ordered}, we deduce the equality
\begin{align}\label{normaB8}
\|\sum_{i=1}^ka_i\mathcal{E}(t_i)\|=\max\{a_i:i\leq k\}
\end{align}
for all $k\in\N$, $T=\{t_1,\ldots,t_k\}\in Bl^k([\N]^8)$ and $(a_i)_{i=1}^k\in[0,1]^k$.

\medskip

Applying Theorem \ref{BrunelSuchestonBBarreras} twice, the first time to $(e_i)_{i\in\N}$ and $([\N]^8,\ldots,[\N]^8,\ldots)$ to obtain $N\in[\N]^{\infty}$, and the second time to $(e_i)_{i\in N}$ and $([N]^2,[N]^2,[N]^8,\ldots,[N]^8,\ldots)$ to obtain $M\in[N]^{\omega}$ such that the sequence $(e_i)_{i\in \N}$ is a $([M]^8,\ldots,[M]^8,\ldots)-$block asymptotic model and a $([M]^2,[M]^2,[M]^8,\ldots,[M]^8,\ldots)-$block asymptotic model of $(e_i)_{\in M}$ under the norms $\|\cdot\|_{8}$ and $\|\cdot\|_{2,2,8}$, respectively, which are defined by 
\begin{align*}
\|\sum_{i=1}^ka_ie_i\|_{8}&=\lim\limits_{ Bl^k([M]^8)\ni \{s_1,\ldots,s_k\}\to\infty}\|\sum_{i=1}^ka_i\mathcal{E}(s_i)\|\\
\|\sum_{i=1}^ka_ie_i\|_{2,2,8}&=\lim\limits_{\ Bl([M]^2,[M]^2,[M]^8,\ldots,[M]^8)\ni\{s_1,\ldots,s_k\}\to\infty}\|\sum_{i=1}^ka_i\mathcal{E}(s_i)\| 
\end{align*}
for all $k\in\N$ and $(a_i)_{i=1}^k\in[-1,1]^k$.
By (\ref{normaB228-2}), (\ref{normaB228-3}), (\ref{normaB2288}), (\ref{normaB2288...}), (\ref{normaB8}) and the definitions of the norms we have
\begin{align}
\|\sum_{i=1}^ka_ie_i\|_{8}=\max\{a_i:i\leq k\}\label{norma8}
\end{align}
and
\begin{align}
&\|\sum_{i=1}^ka_ie_i\|_{2,2,8}=\left\{
\begin{array}{ll}
a_1&\text{if}~a_1\geq a_2\geq\frac{1}{3}a_{\ell}~\text{and}~\frac{3}{2}a_1\geq \frac{9}{8}(a_1+a_2+\frac{2}{3}a_{\ell})\\
\frac{3}{4}(a_1+a_2+\frac{2}{3}a_{\ell})&\text{if}~a_1\geq a_2\geq\frac{1}{3}a_{\ell}~\text{and}~\frac{3}{2}a_1< \frac{9}{8}(a_1+a_2+\frac{2}{3}a_{\ell})\\
a_1&\text{if}~a_1\geq\frac{1}{3}a_{\ell}\geq a_2~\text{and}~\frac{3}{2}a_1\geq \frac{9}{8}(a_1+a_{\ell})\\
\frac{3}{4}(a_1+a_{\ell})&\text{if}~a_1\geq\frac{1}{3}a_{\ell}\geq a_2~\text{and}~\frac{3}{2}a_1<\frac{9}{8}(a_1+a_{\ell})\\
a_2&\text{if}~a_2\geq a_1\geq\frac{1}{3}a_{\ell}~\text{and}~\frac{3}{2}a_2\geq \frac{9}{8}(a_1+a_2+\frac{2}{3}a_{\ell})\\
\frac{3}{4}(a_1+a_2+\frac{2}{3}a_{\ell})&\text{if}~a_2\geq a_1\geq\frac{1}{3}a_{\ell}~\text{and}~\frac{3}{2}a_2< \frac{9}{8}(a_1+a_2+\frac{2}{3}a_{\ell})\\
a_2&\text{if}~a_2\geq\frac{1}{3}a_{\ell}\geq a_1~\text{and}~\frac{3}{2}a_2\geq \frac{9}{8}(a_2+a_{\ell})\\
\frac{3}{4}(a_2+a_{\ell})&\text{if}~a_2\geq\frac{1}{3}a_{\ell}\geq a_1~\text{and}~\frac{3}{2}a_2<\frac{9}{8}(a_2+a_{\ell})\\
a_{\ell}&\text{if}~ \frac{1}{3}a_{\ell}\geq a_1,a_2
\end{array}
\right.\label{norma228}\\
&\qquad\qquad\qquad\qquad\quad\text{where}~a_{\ell}=\max\{a_i:3\leq i\leq k\}~\text{when}~k\geq 3.\notag
\end{align}
for all $k\in\N$ and $(a_i)_{i=1}^k\in[-1,1]^k$.
Due to the inequalities (\ref{desk=2}), (\ref{desk=3}), (\ref{desk=4}) and (\ref{desk}), and the identities  (\ref{norma8}) and (\ref{norma228}) we also conclude
\begin{align*}
\|\sum_{i=1}^ka_ie_i\|_{8}\leq \|\sum_{i=1}^ka_ie_i\|_{2,2,8}\leq2 \|\sum_{i=1}^ka_ie_i\|_{8}
\end{align*}
for each $k\in\N$ and $(a_i)_{i=1}^k\in[-1,1]^k$. Therefore, the basic sequences  $((e_i)_{i\in\N},\|\cdot\|_8)$ and $((e_i)_{i\in\N},\|\cdot\|_{2,2,8})$ are equivalent.

\addvspace{\medskipamount}

We know through Lemma \ref{BModeloAsinDisperso} that $(e_i)_{i\in\N}$ is a spreading sequence under the norm $\|\cdot\|_{8}$. However, if we consider the vectors $e_1+e_2$ and $e_3+e_4$, then we get
\begin{align*}
\|e_1+e_2\|_{2,2,8}&=\lim\limits_{ Bl^2([M]^2,[M]^2)\ni\{s_1,s_2\}\to\infty}\|\mathcal{E}(s_1)+\mathcal{E}(s_2)\|\\
&=\lim\limits_{Bl^2([M]^2,[M]^2)\ni S\to\infty}\frac{3}{2}\\
&=\frac{3}{2}
\end{align*}
and
\begin{align*}
\|e_3+e_4\|_{2,2,8}&=\lim\limits_{ Bl([M]^2,[M]^2,[M]^8,[M]^8)\ni \{s_1,s_2,s_3,s_4\}\to\infty}\|\mathcal{E}(s_3)+\mathcal{E}(s_4)\|\\
&=\lim\limits_{Bl^2([M]^8,[M]^8)\ni\{t_1,t_2\}\to\infty}\|\mathcal{E}(t_1)+\mathcal{E}(t_2)\|\\
&=\lim\limits_{Bl^2([M]^2,[M]^2)\ni T\to\infty}\max\{1,1\}\\
&=1.
\end{align*}
This shows that that $(e_i)_{i\in\N}$ is not a spreading sequence under the norm $\|\cdot\|_{2,2,8}$. Moreover, we have that the norms $\|\cdot\|_{2,2,8}$ and $\|\cdot\|_{8}$ are distinct.

\medskip
In a general form, the reader will find interesting to study the asymptotic models generated by subsequences of $(e_i)_{i\in\N}$ in the completion of $c_{00}$ under the norm $\|\cdot\|_{(m,n)}$ with $m,n\in\N/1$, which is defined by
\begin{align*}
\|x\|_{(m,n)}=\sup&\Big(\big\{|a_i|:i\in\N\big\}\cup\big\{\frac{m+1}{2m}\sum_{i\in s}|a_i|:s\in[\N]^m\big\}\cup\big\{\frac{n+1}{2n}\sum_{i\in s}|a_i|: s\in[\N]^n\big\}\Big)
\end{align*}
for all $x=\sum_{i=1}^{\infty}a_ie_i\in c_{00}$.

\medskip

The example we have described in this section does not answer the question that whether or not all block asymptotic models are equivalent. Formally, this allows us to ask the following:

\begin{question} 
	Is possible to find two non-equivalent block asymptotic models of the same normalized basic sequence? 
\end{question}

From what we have seen in this section, one might apparently conjecture that an infinite sequence of barriers and a ``simple'' permutation of the same one produce equivalent block asymptotic models. Unfortunately we could not deny this conjecture in general, so we propose the following:

\begin{question} 
	Given two distinct barriers $\mathcal{B}_1$ and $\mathcal{B}_2$, is there some relationship between the  $(\mathcal{B}_1,\mathcal{B}_2,\ldots,\mathcal{B}_1,\mathcal{B}_2,\ldots)$-asymptotic model and the $(\mathcal{B}_2,\mathcal{B}_1,\ldots,\mathcal{B}_2,\mathcal{B}_1,\ldots)$-asymptotic model?
\end{question}

 
\end{document}